\documentclass[10pt]{amsart}
\usepackage[active]{srcltx}
\usepackage{a4wide}
\usepackage{amsthm,amsfonts,amsmath,mathrsfs,amssymb}
\usepackage{dsfont}
\usepackage{mathtools}
\usepackage[T1]{fontenc}
\usepackage[utf8]{inputenc}
\usepackage{enumerate}
\usepackage[left=4cm,top=4cm,right=4cm,bottom=4cm]{geometry}

\usepackage{hyperref}

\def\ch{\chi}

\def\D{{\mathbb D}}  \def\T{{\mathbb T}}
  \def\N{{\mathbb N}}

    \def\cb{{\mathcal B}}
    
\def\ch{{\mathcal H}}

\def\({\left(}       \def\){\right)}

\newtheorem{prop}{Proposition}
\newtheorem{lemma}[prop]{Lemma}
\newtheorem{thm}[prop]{Theorem}
\newtheorem{cor}[prop]{Corollary}

\newtheorem{propx}{Proposition}

\newtheorem{lemmax}[propx]{Lemma}

\newtheorem{thmx}[propx]{Theorem}

\begin{document}
\title[Semigroups of c. o. and integral operators on mixed norm spaces]{Semigroups of composition operators and integral operators on mixed norm spaces}
\author[I. Ar\'evalo]{Irina Ar\'evalo}
\address{Departamento de Matem\'aticas, Universidad Aut\'onoma de
Madrid, 28049 Madrid, Spain}
\email{irina.arevalo@uam.es}

\author[M.D. Contreras]{Manuel D. Contreras}
\address{Departamento de Matem\'atica Aplicada II and IMUS, Escuela T\'ecnica Superior de Ingenier\'ia, Universidad de Sevilla,
Camino de los Descubrimientos, s/n 41092, Sevilla, Spain}
\email{contreras@us.es}

\author[L. Rodr\'iguez-Piazza]{Luis Rodr\'iguez-Piazza}
\address{Departmento de An\'alisis Matem\'atico and IMUS, Facultad de Matem\'aticas, Universidad
de Sevilla, aptdo 1160, 41080 Sevilla, Spain}
\email{piazza@us.es}


\date{\today}


\thanks{This research was supported in part by Ministerio de Econom\'{\i}a y Competitividad, Spain,  and the European Union (FEDER) projects MTM2015-63699-P and MTM2015-65792-P, and Junta de Andaluc{\'i}a, FQM133.}

\begin{abstract}
We characterize the semigroups of composition operators that are strongly continuous on the mixed norm spaces $H(p,q,\alpha)$. First, we study the separable spaces $H(p,q,\alpha)$ with $q<\infty,$ that behave as the Hardy and Bergman spaces. In the second part we deal with the spaces $H(p,\infty,\alpha),$ where polynomials are not dense. To undertake this study, we introduce the integral operators, characterize its boundedness and compactness, and use its properties to find the maximal closed linear subspace of $H(p,\infty,\alpha)$ in which the semigroups are strongly continuous. In particular, we obtain that this maximal space is either $H(p,0,\alpha)$ or non-separable, being this result the deepest one in the paper.
\end{abstract}
\maketitle

\section{Introduction}

Let $X$ be a Banach space of analytic functions on the unit disk $\mathbb{D}$. A family $(C_t)_{t\geq 0}$ of bounded operators on $X$ forms a {\sl semigroup} if $C_tf\in X$ for every $f\in X$ and satisfies the following

\begin{enumerate}[1.]
\item $C_0$ is the identity in the space of bounded operators on $X,$
\item $C_{t+s}=C_tC_s,$ for all $t,s\geq0.$
\end{enumerate}

Such a family of bounded operators $(C_t)$ is called {\sl strongly continuous} on a Banach space $X$ if for every $f\in X$ we have $C_tf\in X$ for all $t\geq 0$ and $$\lim_{t\to0^+}\|C_tf-f\|_X=0,$$ and {\sl uniformly continuous} if $$\lim_{t\to0^+}\|C_t-I\|_X=0$$ with $I$ the identity operator.

A {\sl (one-parameter) semigroup of analytic functions} is a family $\{\varphi_t:t\geq0\}$ of analytic self-maps of the disk  that satisfies
\begin{enumerate}[1'.]
\item $\varphi_0$ is the identity in $\mathbb{D},$
\item $\varphi_{t+s}=\varphi_t\circ\varphi_s,$ for all $t,s\geq0,$
\item \label{pointwise1}$\varphi_t\to\varphi_0$ as $t\to0$ uniformly on compact sets of $\mathbb{D}.$
\end{enumerate}

Associated to every semigroup of analytic functions $(\varphi_t)$ we have a semigroup of bounded operators on $X$ given by the composition operator $C_t,$ that is, $C_{t} f=f\circ\varphi_t$ for $f\in X$ and $t\geq0.$ The aim is to understand operator theory properties, such as spectrum, ideals or dynamics of the semigroup of composition operators, in terms of geometric function theory. 
The theory of semigroups of bounded linear operators began with the works of Hille and Yosida (see \cite[Chapter IX]{Yos}). On spaces of analytic functions they were first studied by Berkson and Porta in \cite{BerPor} on Hardy spaces, and later on by Siskakis on Bergman and Dirichlet spaces (\cite{SisBergman} and \cite{SisDirichlet}) (see also his excelent survey \cite{SisSem}). In recent works such as \cite{BCDMS}, \cite{BCDMPS} and \cite{BCHMP} the authors consider semigroups of composition operators on the Bloch space, BMOA and weighted Banach spaces of analytic functions. In the present work, we study the continuity of the semigroup of composition operators $(C_t)$ on the mixed norm spaces.

 For $f$ in the space of analytic functions on the disk $ \mathcal{H}(\mathbb{D})$ and $r\in (0,1)$, let $M_p(r,f)$ be the integral mean $$M_p(r,f)=\left(\int_0^{2\pi}|f(re^{i\theta})|^p\,\frac{d\theta}{2\pi}\right)^{\frac{1}{p}}$$ for $0<p<\infty$ and $$M_\infty(r,f)=\max_{0\leq\theta<2\pi}|f(re^{i\theta})|.$$ 
We consider the spaces $H(p, q, \alpha),$ $0<p,q\leq\infty,$ $0<\alpha<\infty,$ consisting of those analytic functions on $\mathbb{D}$ such that $$\|f\|_{p,q,\alpha}^q=\alpha q\int_0^1(1-r)^{\alpha q-1}M_p^q(r,f)\,dr<\infty,$$ for $q<\infty,$ and $$\|f\|_{p,\infty,\alpha}=\sup_{0\leq r<1}(1-r)^\alpha M_p(r,f)<\infty.$$ 
We also define the ``little-oh'' version, $H(p,0,\alpha),$ as the subspace of functions in $H(p,\infty,\alpha)$ such that $$\lim_{r\to1^-}(1-r)^\alpha M_p(r,f)=0.$$
These integral expressions first appear in Hardy and Littlewood's paper on properties of the integral mean \cite{HL}, but the mixed norm spaces were not explicitly defined until Flett's works \cite{Flett}, \cite{Flett2}. Since then, these spaces have been studied by many authors (see \cite{aj}, \cite{Blasco}, \cite{BKV}, \cite{Gab}, \cite{MP}, \cite{Sledd}). For instance, they appear naturally in the study of coefficient multipliers on Hardy and weighted spaces, and the generalized Hilbert operator on weighted Bergman spaces (see \cite{JVA}, \cite{PR13}, \cite{CGP}, \cite{GGPS} and \cite{PS}). Recently, results on pointwise growth and a characterization of the inclusions between these spaces were given in \cite{Ar}.

The mixed norm spaces form a family of complete spaces that contains the Hardy and Bergman spaces. In particular, one can identify the weighted Bergman space $A^p_\alpha,$ $0<p<\infty,$ $-1<\alpha<\infty,$ of analytic functions on the unit disk such that $$\int_{\mathbb{D}}|f(z)|^p(1-|z|^2)^\alpha\,dA(z)<\infty$$ with the space $H\left(p,p,\frac{\alpha+1}{p}\right),$ and the Hardy space $H^p$ of functions in $\mathcal{H}(\mathbb{D})$ for which $$\sup_{0<r<1}M_p(r,f)<\infty$$ with the limit case $H(p,\infty,0).$ The mixed norm spaces are also related to other spaces of analytic functions, such as Besov and Lipschitz spaces, via fractional derivatives (see \cite[Chapter 7]{JVA}).

 For $q<\infty$, the spaces $H(p,q,\alpha)$  are separable and we will see (Section 4) that the semigroups of composition operators behave similarly to those on Hardy and Bergman spaces, that is, every semigroup of analytic functions induces a strongly continuous semigroup of composition operators but no non-trivial semigroup generates a uniformly continuous semigroup. Nevertheless, for $q=\infty$ polynomials are not dense and we need a different approach. 
We will see that, unlike the integral case, the strong continuity will depend on the particular semigroup. To show this dependence, we define the maximal closed linear subspace of a space $X$ such that the semigroup $(\varphi_t)$ generates a semigroup of operators on it, $$[\varphi_t,X]=\{f\in X:\|f\circ\varphi_t-f\|_X\to0\text{ as }t\to0\}.$$ 
In section 6, we show that
 $$H(p,0,\alpha)\subseteq[\varphi_t,H(p,\infty,\alpha)]\subsetneq H(p,\infty,\alpha)$$ for every semigroup of analytic functions $(\varphi_t).$
When the Denjoy-Wolff point of the semigroup belongs to $\mathbb D$, we characterize when $H(p,0,\alpha)=[\varphi_t,H(p,\infty,\alpha)]$ and prove that if this is  not the case, then $[\varphi_t,H(p,\infty,\alpha)]$ is a non-separable space. On the contrary, when the Denjoy-Wolff point of the semigroup does not belong to the unit disk, then $[\varphi_t,H(p,\infty,\alpha)]$ is always non-separable. The techniques used in this section are, by far, the newest in the area of our work.

In \cite{BCDMPS} the space $[\varphi_t,H(p,\infty,\alpha)]$ was characterized in terms of the integral operator of some function associated to the semigroup. Let us recall that given $g\in \mathcal H (\mathbb D)$, the {\sl integral operator} is defined by 
$$
T_g(f)(z)=\int_0^z f(\zeta)g'(\zeta)\, d\zeta, \quad f\in \mathcal H(\mathbb D).
$$
Therefore, in Section 5 we characterize the boundedness and compactness of the integral operator on mixed norm spaces. Our main result in this sections states that if $g$ belongs to the Bloch space but not to the little Bloch space, then the operator $T_g: H(p,\infty,\alpha)\to H(p,\infty,\alpha)$ fixes 
a copy of $\ell_\infty$. Consequently this last operator has a non separable image.
We also give some applications on inclusions of exponential functions on a Banach space of analytic functions. 


From now on, we will understand $1/\infty$ as zero, the letters $A, B, C, C', K, M$ will be positive constants, and we will say that two quantities are comparable, denoted by $\alpha\approx\beta,$ if there exist two positive constants $C$ and $C'$ such that $$C\alpha\leq\beta\leq C'\alpha.$$

\section{Background on semigroups of analytic functions and composition operators}

As we stated in the introduction, a family $\Phi=\{\varphi_t:t\geq0\}$ of analytic self-maps of the disk $\mathbb{D}$ is a (one-parameter) \textit{semigroup of analytic functions} if it satisfies the following three conditions:
\begin{enumerate}[1.]
\item $\varphi_0$ is the identity in $\mathbb{D},$
\item $\varphi_{t+s}=\varphi_t\circ\varphi_s,$ for all $t,s\geq0,$
\item \label{pointwise}$\varphi_t\to\varphi_0$ as $t\to0$ uniformly on compact sets of $\mathbb{D}.$
\end{enumerate}
In \cite{BerPor} the authors prove the following basic properties of semigroups of analytic functions:
\begin{itemize}
\item If $(\varphi_t)$ is a semigroup, then each map $\varphi_t$ is univalent.
\item The \textit{infinitesimal generator} of $(\varphi_t)$ is the function $$G(z):=\lim_{t\to0^+}\frac{\varphi_t(z)-z}{t},\,z\in\mathbb{D}.$$ This convergence holds uniformly on compact subsets of $\mathbb{D},$ so $G\in\mathcal{H}(\mathbb{D}).$ The generator satisfies $$G(\varphi_t(z))=\frac{\partial\varphi_t(z)}{\partial t}=G(z)\frac{\partial\varphi_t(z)}{\partial z}$$ and characterizes the semigroup uniquely.
\item The function $G$ has a unique representation $$G(z)=(\overline{b}z-1)(z-b)P(z),\,z\in\mathbb{D},$$ where $P\in\mathcal{H}(\mathbb{D})$ with $\text{Re}\,P\geq0$ in $\mathbb{D}$ and $b\in\overline{\mathbb{D}}$ is the \textit{Denjoy-Wolff point} of the semigroup, that is, every self-map of the semigroup shares a common Denjoy-Wolff point $b.$
\item If $(\varphi_t)$ is non-trivial, there exists a unique univalent function $h:\mathbb{D}\to\mathbb{C},$ called the \textit{Koenigs function} of $(\varphi_t)$ such that:
\begin{itemize}
\item If $b\in\mathbb{D}$ then $h(b)=0,$ $h'(b)=1,$ $$h(\varphi_t(z))=e^{G'(b)t}h(z)$$ for $t\geq0,$ $z\in\mathbb{D}$ and $$h'(z)G(z)=G'(b)h(z),$$ $z\in\mathbb{D}.$ 
\item If $b\in\mathbb{T}$ then $h(0)=0,$ $$h(\varphi_t(z))=h(z)+t$$ for $t\geq0,$ $z\in\mathbb{D}$ and $$h'(z)G(z)=1,$$ $z\in\mathbb{D}.$
\end{itemize} 
\end{itemize}

%
%


When the semigroup of analytic functions does not induce a strongly continuous semigroup of composition operators on $X$ we introduce the maximal closed linear subspace of $X$ such that the semigroup $(\varphi_t)$ generates a strongly continuous semigroup of operators on it, denoted by $[\varphi_t,X],$ that is, $$[\varphi_t,X]=\{f\in X:\|f\circ\varphi_t-f\|_X\to0\text{ as }t\to0\}.$$ In \cite{BCDMPS} the authors prove the relation between the subspace $[\varphi_t,X]$ and the infinitesimal  generator of $(\varphi_t).$

\begin{thmx}\label{maxsubspace}
Let $(\varphi_t)$ be a semigroup with generator $G$ and $X$ a Banach space of analytic functions which contains the constant functions and such that $M=\sup_{t\in[0,1]}\|C_t\|_X<\infty.$ Then, $$[\varphi_t,X]=\overline{\{f\in X:Gf'\in X\}}.$$
\end{thmx}

Another useful characterization of $[\varphi_t,X]$ uses the integral operator. Given a semigroup $(\varphi_t)$ with generator $G$ and Denjoy-Wolff point $b,$ we define the function $\gamma:\mathbb{D}\to\mathbb{C},$ called \textit{the associated g-symbol of} $(\varphi_t),$ as $$\gamma(z)=\int_b^z\frac{\zeta-b}{G(\zeta)}\,d\zeta$$ for $b\in\mathbb{D}$ and $$\gamma(z)=\int_0^z\frac{1}{G(\zeta)}\,d\zeta$$ for $b\in\partial\mathbb{D}.$

\begin{propx}\label{maximal}
Let $(\varphi_t)$ be a semigroup with associated g-symbol $\gamma.$ Let $X$ be a Banach space of analytic functions with the properties:
\begin{enumerate}[(i)]
\item $X$ contains the constant functions;
\item For each $b\in\mathbb{D},$ $f\in X\Leftrightarrow\frac{f(z)-f(b)}{z-b}\in X$;
\item If $(C_t)$ is the induced semigroup on $X$ then $\sup_{t\in[0,1]}\|C_t\|<\infty.$
\end{enumerate}
Then $$[\varphi_t, X]=\overline{X\cap(T_\gamma(X)\oplus \mathbb{C})}.$$
\end{propx}

Since we can write the generator function $G$ as $G(z)=(\overline{b}z-1)(z-b)P(z),$ $z\in\mathbb{D},$ with $\text{Re } P\geq0,$ and by composition with an automorphism of the disk, if $b\in\mathbb{D}$ we can assume $b=0,$ so $$\gamma(z)=\int_0^z\frac{\zeta}{G(\zeta)}d\zeta=-\int_0^z\frac{1}{P(\zeta)}d\zeta.$$ If $b\in\partial\mathbb{D},$ without loss of generality we can take $b=1,$ so $G(z)=(1-z)^2P(z)$ and $$\gamma(z)=\int_0^z\frac{1}{(1-\zeta)^2P(\zeta)}d\zeta.$$ From now on we assume that the Denjoy-Wolff point of the semigroup is $b=0,$ if $b\in\mathbb{D},$ or $b=1,$ if $b\in\mathbb{T}.$

%


%

\section{Mixed norm spaces}

%
%
%
Every space $H(p,q,\alpha)$ is complete and, for $p,q\geq 1$, it is a Banach space with the given norm (for $p,q<1$ they are quasi-Banach spaces). Some properties of the functions in these spaces that we will need later on are the following (see \cite{Ar}).

\begin{propx}\label{growth}
Let $0<p,q\leq\infty$ and $0<\alpha<\infty.$ If $f\in H(p,q,\alpha)$ then there exists $C>0$ such that for every $z\in\mathbb{D},$ 
$$|f(z)|\leq \frac{C\|f\|_{p,q,\alpha}}{(1-|z|)^{\alpha+\frac{1}{p}}}.$$
Moreover, for every $z\in\mathbb{D}$ the function $$f_z(w)=\frac{(1-|z|^2)^{\alpha+\frac{1}{p}}}{(1-\bar{z}w)^{2(\alpha+\frac{1}{p})}}$$ belongs to $ H(p,0,\alpha)$ (and therefore $f_z\in H(p,q,\alpha)$ for $0<q\leq\infty$), $\|f_z\|_{p,q,\alpha}\approx 1$ and $|f_z(z)|=\frac{1}{(1-|z|^2)^{\alpha+\frac{1}{p}}}.$

\end{propx}

In particular, if we denote the point-evaluation functional as $\delta_z,$ $\delta_z(f)=f(z)$, then for every $z\in \mathbb{D}$ $$\|\delta_z\|\approx|f_z(z)|=\frac{1}{(1-|z|)^{\alpha+\frac{1}{p}}}.$$

We have the following results for these spaces \cite[Proposition 7.1.3.]{JVA}:
\begin{propx}	\label{Sarason} For $0\leq r<1$, let $f_r(z)=f(rz),$ $z\in\mathbb{D}.$
\begin{itemize}
\item If $f\in H(p,q,\alpha),$ $0<p\leq\infty,$ $0<q,\alpha<\infty,$ then $\|f_r-f\|_{p,q,\alpha}\to0,$ as $r\to1.$
\item If $f\in H(p,0,\alpha),$ $0<p\leq\infty,$ $0<\alpha<\infty,$ then $\|f_r-f\|_{p,\infty,\alpha}\to0,$ as $r\to1.$ 
\end{itemize}
Moreover, if $f\in H(p,\infty,\alpha)$ and $\|f_r-f\|_{p,\infty,\alpha}\to0,$ as $r\to1,$ then $f\in H(p,0,\alpha).$
\end{propx}

A first consequence of Proposition \ref{Sarason} is that polynomials are dense in $H(p,q,\alpha),$ $0<~p\leq\infty,$ $0<q,\alpha<\infty$ and $H(p,0,\alpha),$ $0<p\leq\infty,$ $0<\alpha<\infty.$ The closure in $H(p,\infty,\alpha)$ of the set of all analytic polynomials is $H(p,0,\alpha).$
This result  can be also written in terms of semigroups of composition operators on mixed norm spaces. Namely, if we define $\varphi_t(z)=e^{-t}z,$ for all $t\geq 0$ and $z\in\mathbb{D}$, then $(\varphi_t)$ induces a strongly continuous semigroup of composition operators on $H(p,q,\alpha)$ for $q<\infty$ and $$[\varphi_t,H(p,\infty,\alpha)]=H(p,0,\alpha).$$ 

For our study of semigroups of composition operators, the first thing we need to check is that those operators are bounded on the mixed norm spaces.

\begin{prop}\label{CO} Suppose $0< p,q\leq\infty$ and $0<\alpha<\infty,$ and let $\varphi:\mathbb{D}\to\mathbb{D}$ be an analytic function. Then $C_\varphi$ is bounded on $H(p,q,\alpha)$ and on $H(p,0,\alpha).$ Moreover, it holds that
$$\|C_\varphi\|\lesssim\left(\frac{\|\varphi\|_\infty+|\varphi(0)|}{\|\varphi\|_\infty-|\varphi(0)|}\right)^{\alpha+\frac{1}{p}}.$$ If, in addition,  $\varphi(0)=0$, then $\|C_\varphi\|=1$.
\end{prop}

\begin{proof}
If $\varphi(0)=0,$ then by Littlewood's Subordination Theorem (see \cite[Chapter 1]{Dur}) $M_p(r,f\circ 
\varphi)\leq M_p(r,f),$ and $\|C_\varphi\|\leq1.$ The constant function $f(z)\equiv 1$ shows the equality. 

If $\varphi(0)\neq0,$ we get the bound  arguing as in \cite[Lemma 1]{SisBergman}. Assume firstly that $||\varphi||_\infty=1$.
Fix $0<r<1.$ Applying 
the Schwarz-Pick Lemma, we have 
$$
|\varphi(z)|\leq\frac{|\varphi(0)|+r}{1+|\varphi(0)|r}.$$
Now, since 
$$\frac{|\varphi(0)|+ r}{1+|\varphi(0)|r}\leq\frac{(1-|\varphi(0)|)r+2|\varphi(0)|}{1+|\varphi(0)|}=R
$$ 
we have 
$|\varphi(z)|\leq R$ for $|z|\leq r.$ From here clearly $M_\infty(r,f\circ\varphi)\leq M_\infty(R,f).$
For $0<~p<~\infty$, let $u$ be the harmonic majorant of $|f|^p$ on $|z|\leq R$ such that $u=|f|^p$ on $|z|=R$. 
 Then $|f(z)|^p\leq u(z)$ on $|z|\leq R,$ and hence $|f(\varphi(z))|^p\leq u(\varphi(z))$ for $|z|\leq r.$ From here, 
 $$
 M_p^p(r,f\circ\varphi)=\int_0^{2\pi}|f(\varphi(re^{i\theta}))|^p\,\frac{d\theta}{2\pi}\leq\int_0^{2\pi}u(\varphi(re^{i\theta}))\,\frac{d\theta}{2\pi}=u(\varphi(0)).
 $$ 
 To simplify the notation, let us call 
 $$\frac{ R+|\varphi(0)|}{ R-|\varphi(0)|}=\psi(R),$$ 
 then, by Harnack's inequality $$u(\varphi(0))\leq\psi(R)\,u(0)=\psi(R)\int_0^{2\pi}|f(Re^{i\theta})|^p\,\frac{d\theta}{2\pi}.$$
Therefore, $$M_p(r,f\circ\varphi)\leq\psi(R)^\frac{1}{p}M_p(R,f)$$ for $p<\infty$ and $$M_\infty(r,f\circ\varphi)\leq M_\infty(R,f).$$
Using that $\psi$ is a decreasing function, we obtain
$$
\psi(u)=\frac{ u+|\varphi(0)|}{ u-|\varphi(0)|}\leq\frac{3+|\varphi(0)|}{1-|\varphi(0)|}
$$ for $u\in\left[\frac{2|\varphi(0)|}{1+|\varphi(0)|},1\right]$.

Now we can use these inequalities 
(assuming $1/\infty=0$ so $\psi(R)^{\frac{1}{\infty}}=1$) to bound the norm of $f\circ\varphi$ by the norm of $f$. For $q<\infty,$
\begin{align*}\|f\circ\varphi\|^q_{p,q,\alpha}&=\alpha q\int_0^1(1-r)^{\alpha q-1}\,M_p^q(r,f\circ\varphi)\,dr
\leq\alpha q\int_0^1(1-r)^{\alpha q-1}\,\psi(R)^{\frac{q}{p}}M_p^q(R,f)\,dr.\end{align*}
Next we use the change of variables 
$$
u=R=\frac{(1-|\varphi(0)|)r+2|\varphi(0)|}{1+|\varphi(0)|}
$$ 
getting 
$$
1-r=\frac{1+|\varphi(0)|}{1-|\varphi(0)|}(1-u) \quad \text{and}\quad dr=\frac{1+|\varphi(0)|}{1-|\varphi(0)|}du.
$$
Hence, the norm of the composition operator can be bounded as follows
\begin{align*}
\|f\circ\varphi\|^q_{p,q,\alpha}&\leq \left(\frac{1+|\varphi(0)|}{1-|\varphi(0)|}\right)^{\alpha q}\alpha q\int_{\frac{2|\varphi(0)|}{1+|\varphi(0)|}}^1(1-u)^{\alpha q-1}\,\psi(u)^{\frac{q}{p}}M_p^q(u,f)\,dr
\\&
\leq \left(\frac{1+|\varphi(0)|}{1-|\varphi(0)|}\right)^{\alpha q}\left(\frac{3+|\varphi(0)|}{1-|\varphi(0)|}\right)^{\frac{q}{p}}\alpha q\int_0^1(1-u)^{\alpha q-1}\,M_p^q(u,f)\,du
\\&
\leq C\left(\frac{1+|\varphi(0)|}{1-|\varphi(0)|}\right)^{\alpha q+\frac{q}{p}}\|f\|_{p,q,\alpha}^q.
\end{align*}
And for $H(p,\infty,\alpha),$ just like above,
\begin{align*}
\|f\circ\varphi\|_{p,\infty,\alpha}
&=
\sup_{0<r<1}(1-r)^\alpha M_p(r,f)
\leq\sup_{0<r<1}(1-r)^\alpha\psi(R)^{\frac{1}{p}}M_p(R,f)
\\&
=\left(\frac{1+|\varphi(0)|}{1-|\varphi(0)|}\right)^{\alpha}\sup_{\frac{2|\varphi(0)|}{1+|\varphi(0)|}<u<1}(1-u)^\alpha\left(\frac{ u+|\varphi(0)|}{u-|\varphi(0)|}\right)^{\frac{1}{p}}M_p(u,f)
\\&
\leq C\left(\frac{1+|\varphi(0)|}{1-|\varphi(0)|}\right)^{\alpha+\frac{1}{p}}\|f\|_{p,\infty,\alpha}.
\end{align*}

If $||\varphi||_\infty<1$, writing $U(z)=||\varphi||_\infty z$, then $C_\varphi=C_U\circ C_{\varphi/||\varphi||}$. Applying the previous two cases, we conclude the result. 

The proof for $H(p,0,\alpha)$ is analogous. 
\end{proof}

\section{Semigroups of composition operators on mixed norm spaces given by integrability}

Our first goal is to study the semigroups of composition operators on mixed norm spaces for $q<\infty$. We will see that these spaces behave as the Hardy and Bergman spaces studied in \cite{BerPor} and \cite{SisBergman}.

The next result will help us prove the strong continuity of semigroups of composition operators on Banach spaces of analytic functions where polynomials are dense. The ideas behind the next proposition are taken from \cite{SisSem}, we include its proof for the sake of completeness. 

\begin{prop}\label{general_theorem}  
Let $(\varphi_t)$ be a semigroup of analytic functions in the unit disk and let $X$ be a Banach space of analytic functions such that
\begin{itemize}
\item[(i)] Polynomials are dense in $X$;
\item[(ii)] There is a constant $C>0$ such that if $f$ and $g$ belong to $X$ and $|f|\leq |g|$, then  $\|f\|_X\leq C \|g\|_X$;
\item[(iii)] $M:=\limsup_{t\to 0^+}||C_t||<+\infty$;
\item[(iv)] $\lim_{t\to 0^+}||\varphi_t-\varphi_0||_X=0$.
\end{itemize}
Then the semigroup of operators $(C_t)$ is strongly continuous on $X$.
\end{prop}
\begin{proof}
We have to prove that, given $f\in X,$ it is satisfied that 
$$\|f\circ\varphi_t-f\|_X\to0 \quad \textrm{ as } t\to0.$$ 
Let us fix $n\in\mathbb{N}$. We have that 
\begin{equation*}
\varphi^n_t(z)-z^n=\left(\sum_{k=0}^{n-1}\varphi_t^k(z)z^{n-1-k}\right)(\varphi_t(z)-z).
\end{equation*}
 Hence, $$|\varphi_t^n(z)-z^n|\leq n|\varphi_t(z)-z|.$$ 
 By (ii) and (iv), we deduce that
  $$\|\varphi_t^n-\varphi_0^n\|_X\leq nC\|\varphi_t-\varphi_0\|_X\to0\quad \textrm{ as } t\to0.$$
Then, for every polynomial $p,$ we have that 
\begin{equation}\label{con_polynomial}
\| p\circ\varphi_t-p\|_X\to0.
\end{equation}
Fix $\varepsilon>0$. Since polynomials are dense in $X$ (by (i)),  for every $f\in X$ there exists a  polynomial $p$ such that $||p-f ||_X<\varepsilon$ and 
\begin{align*}
\|f\circ\varphi_t-f\|_X&\leq\|f\circ\varphi_t-p\circ\varphi_t\|_X+\|p\circ\varphi_t-p\|_X+\|p-f\|_X\\&\leq (\|C_{t}\|+1)\|p-f\|_X+\|p\circ\varphi_t-p\|_X\leq(\|C_{t}\|+1)\varepsilon+\|p\circ\varphi_t-p\|_X .
\end{align*}
By (iii) and (\ref{con_polynomial}), 
we get $\limsup_{t\to 0^+}\|f\circ\varphi_t-f\|_X\leq M\varepsilon$. The arbitrariness of $\varepsilon$ implies that $\lim_{t\to 0^+}\|f\circ\varphi_t-f\|_X=0$.
\end{proof}

It is easy to see that, for $q<\infty,$ the mixed norm spaces $H(p,q,\alpha)$ satisfy the axioms of Proposition \ref{general_theorem}.

\begin{thm}\label{finite}
{\rm(a)} Every semigroup of analytic functions generates a strongly continuous semigroup of operators on $H(p,q,\alpha)$ for $q<\infty.$

{\rm(b)} No non-trivial semigroup of analytic functions induces a uniformly continuous semigroup of composition operators on $H(p,q,\alpha)$ for $q<\infty$.
\end{thm}
\begin{proof} a) Let $(\varphi_t)$ be a semigroup of analytic functions in the unit disk. By Proposition~ \ref{CO}, we have that $\limsup_{t\to 0^+}||C_t||<+\infty$. Moreover, 
 since $\varphi_t$ tends to $\varphi_0$ uniformly on compact subsets of the unit disk, $M_p(r,\varphi_t-\varphi_0)\to0,$ and by Lebesgue's Dominated Convergence Theorem, $\|\varphi_t-\varphi_0\|_{p,q,\alpha}\to0.$ Thus, by Proposition \ref{general_theorem}, the semigroup $(C_t)$ is strongly continuous on $H(p,q,\alpha)$.
 
 b) Let us recall that a semigroup of operators is uniformly continuous on $X$ if its infinitesimal generator is a bounded operator on it.
On the other hand, given a semigroup $(\varphi_t)$ with infinitesimal generator $G$, the infinitesimal generator of $(C_t)$ is given by $\Gamma (f)=~Gf'$.
So, suppose  that $\Gamma$ is bounded on $H(p,q,\alpha)$. Then $$\|\Gamma f\|_{p,q,\alpha}=\|Gf'\|_{p,q,\alpha}\leq\|\Gamma\|\|f\|_{p,q,\alpha}$$ for every $f\in H(p,q,\alpha).$ In particular, for $f_n(z)=z^n$ we have $\|\Gamma\|^q\|f_n\|^q_{p,q,\alpha}\geq n^q\|G f_{n-1}\|^q_{p,q,\alpha}.$ Now let $\delta\in (0,1)$ be such that,  for every $n\in  \N$, $$\int_\delta^1(1-r)^{\alpha q-1}r^{(n-1)q}dr\geq\frac{1}{2}\int_0^1(1-r)^{\alpha q-1}r^{(n-1)q}dr.$$ Hence, since the integral means are increasing functions of $r,$ 
\begin{align*}
\|\Gamma\|^q\|f_n\|^q_{p,q,\alpha}&\geq n^q\|G f_{n-1}\|^q_{p,q,\alpha}=\alpha qn^q\int_0^1(1-r)^{\alpha q-1}r^{(n-1)q}M_p^q(r,G)\,dr\\&\geq\alpha q n^q\int_\delta^1(1-r)^{\alpha q-1}r^{(n-1)q}M_p^q(r,G)\,dr\\&\geq\alpha q n^qM_p^q(\delta,G)\int_\delta^1(1-r)^{\alpha q-1}r^{(n-1)q}dr\\&\geq\frac{\alpha q n^q}{2}M_p^q(\delta,G)\int_0^1(1-r)^{\alpha q-1}r^{(n-1)q}dr\\&\geq\frac{\alpha q n^q}{2}M_p^q(\delta,G)\int_0^1(1-r)^{\alpha q-1}r^{nq}dr=\frac{1}{2}n^q\|f_n\|^q_{p,q,\alpha}M^q_p(\delta,G).
\end{align*}
From here, $$nM_p(\delta,G)\leq 2^{\frac{1}{q}}\|\Gamma\|$$ for $n\in\mathbb{N},$ therefore $M_p(\delta,G)=0$ and that means $G\equiv0.$
\end{proof}

Nevertheless, we can not apply the last theorem to $H(p,\infty,\alpha)$ since polynomials are not dense. In the following sections we will deal with this space.

\section{integral operators}

Let $g$ be an analytic function on the unit disk and $T_g$ the {\sl integral operator} induced by it, namely, $$T_g(f)(z)=\int_0^z f(\zeta)g'(\zeta)\,d\zeta,\quad f\in \mathcal H(\mathbb D)$$ for any $z\in\mathbb{D}.$ This operator was first studied by Pommerenke on the Hardy space $H^2$ (see \cite{Pom}) and later on Hardy and Bergman spaces (see, for example the papers \cite{ASHp}, \cite{ASAp} by Aleman and Siskakis or the recent one \cite{CPPR} by the second author, Pel\'aez, Pommerenke and R\"atty\"a). In the case $g(z)=z$, the integral operator $T_g$ is the classical {\sl Volterra operator} and we denote it by $V$ throughout the paper. 

The boundedness and compactness of the integral operator on the mixed norm spaces will be given in terms of two well-known spaces of analytic functions: the Bloch spaces. We say that an analytic function on the unit disk $f$ is in the \textit{Bloch space} $\mathcal{B}$ if and only if $$\sup_{z\in\mathbb{D}}(1-|z|^2)\,|f'(z)|<\infty,$$ and in the \textit{little Bloch space} $\mathcal{B}_0$ if and only if $$\lim_{|z|\to1^-}(1-|z|^2)\,|f'(z)|=0.$$
Observe that $f\in \cb$ (resp. $f\in \cb_0$) if and only if $f'\in H(\infty,\infty,1)$ (resp. $f'\in H(\infty,0,1)$).

The following lemma, based on a theorem by Hardy and Littlewood, will allow us to relate the integral operator $T_g$ with the {\sl pointwise multiplier} $M_{g'},$ that is, $M_{g'} f=g'f$ for $f\in X.$ 

\begin{lemmax}\label{derivative} \cite[Thm.~5.5]{Dur}
For $0< p\leq\infty$, $\alpha>0$ and $f\in \mathcal H(\D)$ it is satisfied:
\begin{enumerate}[{(1)}]
\item\label{HLO} $M_p(r,f)=O(1-r)^{-\alpha}$ if and only if $M_p(r,f')=O(1-r)^{-(\alpha+1)}$, that is, $f\in H(p,\infty,\alpha)$ if and only if $f'\in H(p,\infty,\alpha+1)$.
\item\label{HLo} $M_p(r,f)=o(1-r)^{-\alpha}$ if and only if $ M_p(r,f')=o(1-r)^{-(\alpha+1)}$,  that is, $f\in H(p,0,\alpha)$ if and only if $f'\in H(p,0,\alpha+1)$.
\end{enumerate}
\end{lemmax}

A part of our results is based on the following theorem due to Aleman and Siskakis.

\begin{thmx}\label{AlSis}\cite{ASAp}
Let $X$ a Banach space of analytic functions such that \begin{enumerate}
\item the point evaluation functionals $\delta_z$ are bounded for every $z\in\mathbb{D},$
\item  for every $\lambda\in\mathbb{T}$ the operator $U_\lambda f(z)=f(\lambda z)$ is bounded and $\sup_{\lambda\in\mathbb{T}}\|U_\lambda\|_X<\infty,$ and
\item for some $s\in(0,1)$ the composition operator $C_{\psi}f=f\circ\psi_s$ with $\psi_s(z)=sz+1-s$ is bounded on $X$.
\end{enumerate}  
Let $g$ be an analytic function on the unit disk.
\begin{enumerate}
\item[{\rm(a)}] If $T_g:X\to X$ is bounded then $g\in\mathcal{B}$ and there is a constant $C$ (which does not depend on $g$) such that  $\|g\|_\mathcal{B}\leq C\|T_g\|.$
\item[{\rm(b)}] Moreover, if the multiplication operator $M_z f(z)=zf(z)$ is bounded on $X$ and, for some fixed $t\in(0,1)$, it is satisfied that $\lim_{n\to\infty}\|(tz+1-t)^nf(z)\|=0$ for all $f\in X$, then $T_g$ being compact on $X$ implies $g\in\mathcal{B}_0.$
\end{enumerate}
\end{thmx}

With the last two results we are ready to characterize the symbols for which the integral operator is bounded and compact on $H(p,q,\alpha)$ for $q<\infty.$ For $q=\infty$, the spaces $H(p,0,\alpha)$ and $H(p,\infty,\alpha)$ do not satisfy the hypothesis (b) above, so we will need a different and deeper argument.

\begin{prop}
Let $g$ be an analytic function on the unit disk, $T_g$ the integral operator that induces, and $0<p\leq\infty,$ $0<\alpha,q<\infty.$
\begin{itemize}
\item $T_g:H(p,q,\alpha)\to H(p,q,\alpha)$ is bounded if and only if $g\in\mathcal{B}.$
\item $T_g:H(p,q,\alpha)\to H(p,q,\alpha)$ is compact if and only if $g\in\mathcal{B}_0.$
\end{itemize}
\end{prop}

\begin{proof}
The sufficiency in both statements is given by Theorem \ref{AlSis}. 

If $g\in\mathcal{B}$ then the multiplication operator $M_{g'}$ is bounded from $H(p,q,\alpha)$ to $H(p,q,\alpha+~1)$ and (\ref{HLO}) in Lemma~\ref{derivative} shows that $V(f)(z)=\int_0^zf(\zeta)\,d\zeta$ is a bounded operator from $H(p,q,\alpha+1)$ to $H(p,q,\alpha),$ so $T_g=V\circ M_{g'}$ is bounded. 

If $g\in\mathcal{B}_0,$ using again that $T_g=V\circ M_{g'},$ and that $V$ is bounded, the result follows from the compactness of the multiplication operator $M_{g'}:H(p,q,\alpha)\to H(p,q,\alpha+1)$. Indeed, let $(f_n)$ be a sequence in the unit ball of $H(p,q,\alpha)$ that converges to zero uniformly on compact subsets of the unit disk. We have to prove that $||M_{g'}f_n||_{p,q,\alpha}\to 0$ as $n\to \infty$. Since $g\in\mathcal{B}_0,$ for $\varepsilon>0$ there exists a $R<1$ such that $|g'(z)|(1-|z|)<\varepsilon$ for $|z|\geq R.$ Now, let $N_0\in \N$ be such that $|f_n(z)|\leq \varepsilon/||g||_\cb$ for $n\geq N_0$ and $|z|\leq R$. Then 
\begin{align*}
\|g'f_n\|_{p,q,\alpha+1}^q&=(\alpha+1) q\int_0^1(1-r)^{(\alpha+1)q-1}M_p^q(r,g'f_n)\,dr\\&
\leq(\alpha+1) q\int_0^1(1-r)^{(\alpha+1)q-1}\left(\sup_{\theta\in [0,2\pi]}|g'(re^{i\theta})|\right)^qM_p^q(r,f_n)\,dr\\&
\leq\|g\|_\mathcal{B}^q\,(\alpha+1) q\int_0^R(1-r)^{\alpha q-1}M_p^q(r,f_n)\,dr\\&
+\varepsilon^q\,(\alpha+1) q\int_R^1(1-r)^{\alpha q-1}M_p^q(r,f_n)\,dr\leq 2\frac{\alpha+1}{\alpha}\varepsilon^q.
\end{align*}

Then $\|g'f_n\|_{p,q,\alpha}\to 0$ and therefore $M_{g'}:H(p,q,\alpha)\to H(p,q,\alpha+1)$ is compact.
\end{proof}

Now we prove directly the boundedness on $H(p,\infty,\alpha)$ and $H(p,0,\alpha).$

\begin{prop}\label{integralO} Let $g$ be an analytic function in the unit disk. The following are equivalent:
\begin{itemize}
\item[(a)] $T_g: H(p,\infty,\alpha)\to H(p,\infty,\alpha)$ is bounded;
\item[(b)] $T_g: H(p,0,\alpha)\to H(p,0,\alpha)$ is bounded;
\item[(c)] $T_g: H(p,0,\alpha)\to H(p,\infty,\alpha)$ is bounded;
\item[(d)] $g\in\mathcal{B}.$
\end{itemize}
\end{prop}

\begin{proof}
If $g\in\mathcal{B}$ and $f\in H(p,\infty,\alpha)$, then 
\begin{align*}
\|g'f\|_{p,\infty,\alpha+1}&=\sup_{0\leq r<1}(1-r)^{\alpha+1}M_p(r,g'f)\\&\leq ||g||_\mathcal B\sup_{0\leq r<1}\frac{(1-r)^{\alpha+1}}{1-r}M_p(r,f)= ||g||_\mathcal B\,\|f\|_{p,\infty,\alpha}.
\end{align*}
So, by Lemma \ref{derivative}, (d) implies (a) and (c). The same inequalities show that if $f\in~H(p,0,\alpha)$, then $g'f\in  H(p,0,\alpha+1)$. Again, by Lemma \ref{derivative}, (d) implies (b).

On the other hand, suppose $g'f\in H(p,\infty,\alpha+1)$ for every $f\in H(p,0,\alpha).$ This means that the operator $M_{g'}$ is bounded from  $H(p,0,\alpha)$ into $H(p,\infty,\alpha+1)$. Let us denote by $M$ the norm of this operator.  Then, by Proposition \ref{growth}, there exists a constant $C>0$ such that, for every $z\in\mathbb{D},$
$$
|g'(z)f(z)|\leq\frac{C\|g'f\|_{p,\infty,\alpha+1}}{(1-|z|)^{\alpha+1+\frac{1}{p}}}\leq\frac{CM\|f\|_{p,\infty,\alpha}}{(1-|z|)^{\alpha+1+\frac{1}{p}}}.
$$ 
Choosing $f_z\in H(p,0,\alpha)$ as $$f_z(w)=\frac{(1-|z|^2)^{\alpha+\frac{1}{p}}}{(1-\overline{w}z)^{2(\alpha+\frac{1}{p})}},$$ a function that satisfies $|f_z(z)|=(1-|z|^2)^{-(\alpha+\frac{1}{p})}$ and $\|f_z\|_{p,\infty,\alpha}\approx1,$ we get $$|g'(z)f_z(z)|=\frac{|g'(z)|}{(1-|z|^2)^{\alpha+\frac{1}{p}}}\lesssim\frac{CM}{(1-|z|)^{\alpha+1+\frac{1}{p}}}.$$ From here it is clear that $g\in\mathcal{B}.$ This argument shows that (b) or (c) implies (d).
\end{proof}

The last case, that is, the boundedness of the integral operator from the bigger space $H(p,\infty,\alpha)$ to the smaller space $H(p,0,\alpha)$ will be very interesting for us in the study of semigroups of composition operators on $H(p,\infty,\alpha).$ 

We will need two lemmas to continue. As usual, 
given an analytic function $f\in \ch(\D)$, we denote by 
$\widehat f(k)$ its $k$'th Taylor coefficient. 

\begin{lemma}\label{Fejer}
For every $N\in\N$ there exists a polynomial $G_N$ satisfying:
\begin{itemize}
\item[(a)] $\widehat{G_N}(k)\ge 0$, for every $k\in [N,3 N]$, 
\item[(b)] $\widehat{G_N}(k)=0$, for every $k\notin (N,3 N)$,
\item[(c)] $\|G_N\|_{H^1}=1$, and 
\item[(d)] $\|G_N\|_{H^\infty}=\sum_{k=N}^{3N} \widehat{G_N}(k)=N$.
\item[(e)] For $1<p<+\infty$, $\|G_N\|_{H^p}\le N^{\frac{1}{p'}}$, where $\frac{1}{p'} + \frac{1}{p}=1$.
\end{itemize}
\end{lemma}
\begin{proof}
Let $F_{N-1}$ be the $(N-1)$'th Fej\'er Kernel, namely
$$
F_{N-1}(e^{it})=\sum_{k=-N+1}^{N-1} \Bigl(1-\frac{|k|}{N}\Bigr) e^{ikt}\,, \qquad e^{it}\in \T\,.
$$
It is known that $\|F_{N-1}\|_{L^1(\T)}=1$. Define $G_N(z)=z^{2N}F_{N-1}(z)$ for 
$z\in \T$, and we obtain the required polynomial. All the properties are easy to check.
The last one comes from the inequality $\|f\|_p\le \|f\|_\infty^{\frac{1}{p'}}\|f\|_1^{\frac{1}{p}}$.
\end{proof}

\begin{lemma}\label{Onoo}
Suppose $g\in \cb \setminus \cb_0$. Then there exist  $\delta\in (0,\pi/8)$,
an increasing sequence $(r_n)_n$ in $(0,1)$, and a sequence $(t_n)_n$ in $ [0,2\pi)$ satisfying:
\begin{itemize}
\item[(a)] For every $n\in\N$ and every $t\in [-\delta(1-r_n),\delta(1-r_n)]$ we have
$$
(1-r_n)|g'(r_n e^{i(t_n+t)})|\ge \delta.
$$
\item[(b)] $\lim_{n\to\infty} r_n =1$.
\end{itemize}
\end{lemma}
\begin{proof} Write $M=||g||_{\mathcal B}$.
Since $g\notin \cb_0$ we get the existence of $\eta>0$, an increasing sequence $(r_n)_n$ in 
$(0,1)$ with $\lim_{n\to\infty} r_n=1$, and a sequence $(t_n)_n$ in $[0,2\pi)$, such that
\begin{equation}\label{g_no_Bloch0}
(1-r_n)|g'(r_ne^{it_n})|\ge \eta\,,\qquad\text{for all $n$.}
\end{equation}

From  the Maximum Principle we have 
$$
|g'(z)|\le \frac{2M}{1-r_n}, \qquad \text{if }\,|z|\le\frac{1+r_n}{2}.
$$
By Cauchy's inequality this gives 
$$
|g''(z)|\le \frac{4M}{(1-r_n)^2},  \qquad \text{if }\,|z|\le r_n.
$$
Take $\delta>0$ and $|s-t_n|<\delta(1-r_n)$. Then
$$
|g'(r_ne^{is})-g'(r_ne^{it_n})|\le |r_ne^{is} - r_ne^{it_n}|\frac{4M}{(1-r_n)^2}\le \frac{4|s-t_n|M}{(1-r_n)^2}
\le \frac{4\delta M}{1-r_n}
\le \frac{\eta/2}{1-r_n}\,,
$$
if $\delta$ is small enough.
This fact and (\ref{g_no_Bloch0}) imply, for $|t|\le \delta(1-r_n)$,
$$
|g'(r_ne^{i(t+t_n)})|(1-r_n)\ge \frac{\eta}{2} \ge \delta\,,
$$
if $\delta$ is small enough. The lemma follows.
\end{proof}


Now we are ready to state the main result of this section.

\begin{thm}\label{copy_l_infty}If $g\in \cb \setminus \cb_0$, then the operator $T_g: H(p,0,\alpha)\to H(p,0,\alpha)$ fixes 
a copy of $c_0$ and the operator $T_g: H(p,\infty,\alpha)\to H(p,\infty,\alpha)$ fixes 
a copy of $\ell_\infty$. Consequently this last operator has a non separable image.
\end{thm}

\begin{proof} Since $g\in \cb$, we have that 
$T_g: H(p,\infty,\alpha)\to H(p,\infty,\alpha)$ is bounded. 
Once again we write $T_g$ as $V\circ M_{g'},$ with $M_{g'}\colon H(p,\infty,\alpha )\to H(p,\infty,\alpha+1)$ the operator of multiplication by $g'$ and $V\colon H(p,\infty,\alpha+1)\to H(p,\infty,\alpha)$ the Volterra operator.
We know that, if we avoid the constant functions, $V$ is an isomorphism.
More concretely, if we call $X$ the one-codimensional subspace of $H(p,\infty,\alpha)$ defined by
$$
X=\{f\in H(p,\infty,\alpha) : f(0)=0\}\,,
$$
then $V\colon H(p,\infty,\alpha+1)\to X$ is an onto isomorphism whose inverse is the derivation.
Therefore we only need to prove that $M_{g'}\colon H(p,\infty,\alpha)\to H(p,\infty,\alpha+1)$ 
fixes a copy of $\ell_\infty$ if $g\in \cb \setminus \cb_0$.

To do this we need to construct a bounded linear operator
$\Phi\colon\ell_\infty \to H(p,\infty,\alpha)$ such that, for certain $C>0$,
\begin{equation}\label{fijarcopia}
C \|M_{g'}\left(\Phi {\bf a}\right)\|_{p,\infty,\alpha+1} \ge \|{\bf a}\|_{\ell_\infty}\,,
\qquad \text{for all ${\bf a}\in \ell_\infty$.}
\end{equation}
Apply Lemma~\ref{Onoo} to $g$ and fix $\beta>0$ big enough (this $\beta$ will depend on 
$||g||_{\mathcal B}$, $\delta$, $\alpha$ and $p$).
Passing to a subsequence if necessary, we can assume that $r_n\ge 1/2$, for all $n$ and
\begin{equation}\label{loserres}
\frac{1-r_n}{1-r_{n+1}}\ge \beta\,,\qquad \text{for all $n$.}
\end{equation}

Now consider a sequence $(N_n)_n$ of positive integers such that 
\begin{equation}\label{enerre}
N_n(1-r_n)\in [1,2]\,, \qquad\text{for all $n\in\N$.}
\end{equation}
By \eqref{loserres} and \eqref{enerre} we have,  if $\beta$ is big enough,
$$
\frac{N_{n+1}}{N_n}\ge \frac{\beta}{2} \ge 3\,, \qquad\text{for all $n\in\N$.}
$$
Let $	\nu=\alpha +\frac{1}{p} -1=\alpha-\frac{1}{p'}$.
For every $n\in\N$ define the function  $g_n$ by
$$
g_n(z)= N_n^{\nu}G_{N_n}(e^{-it_n} z)\,,\qquad z\in\D.
$$
The $G_N$'s are given in Lemma~\ref{Fejer} and the $t_n$'s in Lemma~\ref{Onoo}. 

Observe that, for every $r\in (0,1)$ and every $t\in [0,2\pi)$, we have
\begin{equation}\label{cotauniforme}
|g_n(re^{it})|\le N_n^{\nu} \sum_{k=N_n}^{3N_n} \widehat{G_{N_n}}(k)r^{N_n}=r^{N_n} N_n^{1+\nu}
= e^{N_n\log r + (\alpha+\frac{1}{p})\log N_n}\,.
\end{equation}
This and the fact that $N_n\ge 3^{n-1}$ yield that, for every ${\bf a}=(a_n)_n\in\ell_\infty$, the series 
$$
\sum_{n=1}^\infty a_n g_n(z)
$$
converges uniformly on compact subsets of $\D$ and its sum defines a function $\Phi{\bf a}$
holomorphic on $\D$. 

Let us see that $\Phi{\bf a}$ belongs to $H(p,\infty,\alpha)$ for all ${\bf a}\in\ell_\infty$. Since 
$\widehat{g_n}(k)=0$, for $k\le N_n$,
we have the estimate, for $r\in (0,1)$ and $n\in\N$,
\begin{equation*}
M_p(r,g_n)\le N_n^{\nu} r^{N_n} \|G_{N_n}\|_{H^p} \le r^{N_n} N_n^{\nu+\frac{1}{p'}} =
r^{N_n} N_n^{\alpha} .
\end{equation*}
Consequently, for ${\bf a}\in\ell_\infty$, we have
\begin{equation*}
M_p(r,\Phi{\bf a})\le \|{\bf a}\|_{\ell_\infty} \sum_{n=1}^\infty M_p(r,g_n)\le 
 \|{\bf a}\|_{\ell_\infty} \sum_{n=1}^\infty r^{N_n} N_n^{\alpha} \,.
\end{equation*}

Take $r\in (0,1)$ and, putting $r_0=0$, define $l\in\N$ by the condition $r_{l-1}<r\le r_l$.
Thus, in the case $l\ge 2$, we have
\begin{align*}
(1-r)^{\alpha}\sum_{n=1}^{l-1} r^{N_n} N_n^{\alpha} &\le (1-r_{l-1})^{\alpha}\sum_{n=1}^{l-1} N_n^{\alpha}   
=[(1-r_{l-1})N_{l-1}]^{\alpha} \sum_{n=1}^{l-1} \left(\frac{N_n}{N_{l-1}}\right)^{\alpha} \\
&\le 2^{\alpha}\sum_{k=0}^\infty \left(\frac{\beta}{2}\right)^{-k\alpha}:=A<+\infty\,.
\end{align*}
For $n\ge l$ we have, using $\log r\le r-1$ and $N_n(1-r)\ge N_n(1-r_n)\ge 1$,
\begin{align*}
\log\bigl[ (1-r)^{\alpha} r^{N_n} N_n^{\alpha}\bigr]&\le 
\alpha\log\bigl[N_n(1-r)\bigr]+ N_n(r-1) \\
&\le C_\alpha-\frac{N_n(1-r)}{2} \le C_\alpha-\frac{N_n(1-r_l)}{2}\,,
\end{align*} 
for certain $C_\alpha>0$ satisfying $\alpha\log x\le C_\alpha + \frac{x}{2}$, for all $x\ge 1$.
Therefore
$$
\sum_{n=l}^\infty (1-r)^{\alpha} r^{N_n} N_n^{\alpha} \le \sum_{n=l}^\infty e^{C_\alpha} e^{-\frac{\beta^{n-l}}{2}}
=e^{C_\alpha}\sum_{j=0}^\infty e^{-\frac{\beta^j}{2}}:=B<+\infty\,.
$$
Putting all together we have, for all $r\in(0,1)$, 
$$
(1-r)^{\alpha}M_p(r,\Phi{\bf a})\le (A+B)\|{\bf a}\|_{\ell_\infty}\,.
$$
Taking the supremum over $r$ we see that $\Phi{\bf a}\in H(p,\infty,\alpha)$
and $\Phi\colon\ell_\infty\to H(p,\infty,\alpha)$ is a bounded linear operator.

It remains to prove \eqref{fijarcopia}. We can assume ${\bf a}=(a_n)_n\in  \ell_\infty$ and
 $\|{\bf a}\|_{\ell_\infty}=1$. Pick $l\in \N$ such that
$|a_l|\ge 1/2$. We are going to prove that, for certain $\eta>0$ we have
\begin{equation}\label{inferior}
|\Phi{\bf a}(r_le^{i(t+t_l)})|\ge \eta (1-r_l)^{-\frac{1}{p}-\alpha}\,,\qquad
\text{for all $t\in [-\delta(1-r_l),\delta(1-r_l)]$.}
\end{equation}
This and Lemma~\ref{Onoo}(a) yield 
$$
|M_{g'}(\Phi{\bf a})(r_le^{i(t+t_l)})|\ge \eta\delta (1-r_l)^{-\frac{1}{p}-1-\alpha}\,,\qquad
\text{for all $t\in [-\delta(1-r_l),\delta(1-r_l)]$,}
$$
and consequently we get \eqref{fijarcopia} since
$$
 \|M_{g'}\bigl(\Phi {\bf a}\bigr)\|_{p,\infty,\alpha+1}\ge
 (1-r_l)^{1+\alpha}M_p(r_l,M_{g'}(\Phi{\bf a}))\ge \eta\delta\left(\frac{\delta}{\pi}\right)^{\frac{1}{p}}:=\frac{1}{C}\,.
$$

Let us prove \eqref{inferior}. 
By the definition of $\Phi{\bf a}$, for every $t\in \mathbb{R}$, we have
\begin{equation}\label{mamapata}
\begin{split}
|\Phi{\bf a}(r_le^{i(t+t_l)})|
&\ge |a_l||g_l(r_le^{i(t+t_l)})|-\sum_{n\ne l}|a_n||g_n(r_le^{i(t+t_l)})| \\
&\ge \frac{1}{2}\; \bigl|g_l(r_le^{i(t+t_l)})\bigr| -\sum_{n\ne l} M_\infty(r_l,g_n)\,.
\end{split}
\end{equation}
By \eqref{enerre},
for  $t\in [-\delta(1-r_l),\delta(1-r_l)]$ and $k\le 3N_l$, we have $|kt|\le 6\delta\le \pi/3$.
and $\cos(kt)\ge 1/2$. Therefore
\begin{equation}\label{patito1}
\begin{split}
\bigl|g_l(r_le^{i(t+t_l)})\bigr| &\ge \text{Re}\bigl(g_l(r_le^{i(t+t_l)})\big)
= N_l^{\nu}\sum_{k=N_l}^{3N_l} r_l^k \widehat{G_{N_l}}(k)\cos(kt) \\
&\ge \frac{N_l^\nu r_l^{3N_l}}{2} \sum_{k=N_l}^{3N_l} \widehat{G_{N_l}}(k)\ge \kappa N_l^{\nu+1} \,,
\end{split}
\end{equation}
for certain $\kappa>0$.

Using \eqref{cotauniforme} we can estimate $M_\infty(r_l,g_n)\le 2r_l^{N_n} N_n^{\nu+1}$. Thus, 
\begin{equation}\label{patito2}
\begin{split}
\sum_{n=1}^{l-1}M_\infty(r_l,g_n)&\le \sum_{n=1}^{l-1}2N_n^{\nu+1}\le 2N_l^{1+\nu}\sum_{n=1}^{l-1}\left(\frac{2}{\beta}\right)^{(l-n)(1+\nu)}\\
&\le 4 \left(\frac{2}{\beta}\right)^{\nu+1} N_l^{\nu+1}\le \frac{\kappa N_l^{\nu+1}}{4}\,,
\end{split}
\end{equation}
if $\beta$ is big enough.

We use $r_l^{N_l}= \exp( N_l\log r_l  )\le \exp\bigl( N_l(r_l-1)   \bigr) \le e^{-1}$, to obtain, if $l<n$,
$$
M_\infty(r_l,g_n)\le 2r_l^{N_n} N_n^{\nu+1}\le 2N_l^{\nu+1}  \Bigl(\frac{N_n}{N_l}\Bigr)^{\nu+1} \exp\Bigl(-\frac{N_n}{N_l}\Bigr) \,.
$$
If $\beta$ is big enough we have $x^{-1}\ge x^{1+\nu}e^{-x}$ for $x\ge \beta$. Thus 
\begin{equation}\label{patito3}
\sum_{n=l+1}^\infty M_\infty(r_l,g_n)\le2N_l^{\nu+1} \sum_{n=l+1}^\infty \Bigl(\frac{N_l}{N_n}\Bigr)   \le 2N_l^{1+\nu}\sum_{n=l+1}^\infty\beta^{l-n}
\le \frac{\kappa N_l^{\nu+1}}{4}\,,
\end{equation}
if $\beta$ is big enough.

Finally we collect all the estimates. For $t\in [-\delta(1-r_l),\delta(1-r_l)]$, 
by \eqref{mamapata}, \eqref{patito1}, \eqref{patito2} and \eqref{patito3},
\begin{equation*}
\begin{split}
|\Phi{\bf a}(r_le^{i(t+t_l)})|& \ge \frac{1}{2}\; \bigl|g_l(r_le^{i(t+t_l)})\bigr| -\sum_{n=1}^{l-1} M_\infty(r_l,g_n)-\sum_{n=l+1}^\infty M_\infty(r_l,g_n) \\
&\ge \left(\kappa- \frac{\kappa}{4} - \frac{\kappa}{4}\right) N_l^{1+\nu}=\frac{\kappa}{2} N_l^{\alpha + \frac{1}{p}}
\ge \frac{\kappa}{2^{\alpha+\frac{1}{p}+1}} (1-r_l)^{-\alpha-\frac{1}{p}}\,.
\end{split}
\end{equation*}
We have proved \eqref{inferior} and thus the theorem follows for the space $H(p,\infty,\alpha)$. Since the functions $g_n$ are polynomials, a similar argument shows that if ${\bf a}\in c_0$, then $\Phi{\bf a} \in H(p,0,\alpha)$. This finishes the proof.   
\end{proof}

%

%

With this we are now ready to characterize the boundedness of $T_g: H(p,\infty,\alpha)\to H(p,0,\alpha)$ and the compactness of $T_g$ on $H(p,\infty,\alpha)$ and on $H(p,0,\alpha).$ All of them are equivalent to $g\in\mathcal{B}_0.$

\begin{cor}\label{integral_compact}  Let $g$ be a function in the Bloch space $\mathcal B$.  The following are equivalent:
\begin{itemize}
\item[(a)] $T_g: H(p,\infty,\alpha)\to H(p,\infty,\alpha)$ is compact;
\item[(b)] $T_g: H(p,\infty,\alpha)\to H(p,\infty,\alpha)$ is weakly compact;
\item[(c)] $T_g: H(p,\infty,\alpha)\to H(p,0,\alpha)$ is bounded;
\item[(d)] $T_g: H(p,0,\alpha)\to H(p,0,\alpha)$ is compact;
\item[(e)] $T_g: H(p,0,\alpha)\to H(p,0,\alpha)$ is weakly compact;
\item[(f)] $T_g: H(p,\infty,\alpha)\to H(p,\infty,\alpha)$ does not fix a copy of $\ell_\infty$;
\item[(g)] $T_g: H(p,0,\alpha)\to H(p,0,\alpha)$ does not fix a copy of $c_0$;
\item[(h)] $g\in\mathcal{B}_0.$
\end{itemize}
\end{cor}

\begin{proof}
Let us assume that $g\in\mathcal{B}_0.$
Take $(f_n)$ a sequence in the unit ball of $H(p,\infty,\alpha)$ that goes to zero uniformly on compact subsets of the unit disk. Let us fix $\varepsilon>0$. Then there is $R<1$ such that $|g'(z)|(1-|z|)<\varepsilon$ whenever $|z|\geq R$. Moreover, there is $N_0\in \N$ such that if $n\geq N_0$ we have that $|f_n(z)|\leq \varepsilon/||g||_\cb$ for all $|z|\leq R$. 
On the one hand, if $r\leq R$, then 
\begin{align*}
(1-r)^{\alpha+1}M_p(r,g'f_n)&\leq(1-r)^{\alpha+1}\sup_{\theta\in [0,2\pi]}|g'(re^{i\theta})|M_p(r,f_n)\\
&\leq ||g||_\cb (1-r)^{\alpha}M_p(r,f_n)\leq \varepsilon.
\end{align*}
On the other hand, if $r>R$, then 
\begin{align*}
(1-r)^{\alpha+1}M_p(r,g'f_n)&\leq(1-r)^{\alpha+1}\sup_{\theta\in [0,2\pi]}|g'(re^{i\theta})|M_p(r,f_n)\\
&\leq ||f_n||_{p,\infty, \alpha} (1-r)\sup_{\theta\in [0,2\pi]}|g'(re^{i\theta})|\leq \varepsilon.
\end{align*}
Thus, $\lim ||g'f_n||_{p,\infty, \alpha+1}=0$. This implies that $M_{g'}:  H(p,\infty,\alpha)\to H(p,\infty,\alpha+1)$ and $M_{g'}:  H(p,0,\alpha)\to H(p,0,\alpha+1)$ are compact. So, (h) implies (a) and (d). 

A similar argument shows that (h) implies (c). 

Being trivial that (a) $\Rightarrow$ (b) $\Rightarrow$ (f), that (c) $\Rightarrow $ (f) (since $H(p,0,\alpha)$ is separable) and that (d) $\Rightarrow$ (e) $\Rightarrow$ (g), we only have to prove that both (f) and (g) imply (h). But this  is just Theorem \ref{copy_l_infty}.
\end{proof}

To conclude this section on the integral operator, we give an application to the inclusion of exponential functions in the space. 

\begin{prop}
Let $X$ be a Banach space of analytic functions and $g$ an analytic function on the unit disk. 
\begin{itemize}
\item If $T_g:X\to X$ is bounded, then $e^{sg}\in X$ for some $s>0.$
\item If $T_g:X\to X$ is compact, then $e^{sg}\in X$ for every $s>0.$
\end{itemize}
\end{prop}

\begin{proof}
Let $g$ be an analytic function on $\mathbb{D}$ such that $T_g:X\to X$ is bounded and suppose $g(0)=0.$ Then, $T_g(1)=g$ and in general $T_g^n(1)=\frac{1}{n!}g^n.$ Let $0<s<1/r(T_g)$ where $r(T_g)$ is the spectral radius of the operator $T_g.$ Hence, $\sum_{n=0}^\infty s^n T_g^n$ converges in the operator norm. With $f\equiv 1$ we have that $$\sum_{n=0}^\infty s^n T_g^n(1)=\sum_{n=0}^\infty \frac{s^ng^n}{n!}=e^{sg}\in X$$ for some $s>0.$

If $T_g:X\to X$ is compact, its spectrum is the set of its eigenvalues and $\{0\}.$ Let $\lambda\neq0,$ then $T_gf=\lambda f$ implies $f(0)=0.$ Differentiating we get $f(z)g'(z)=\lambda f'(z),$ that is, $f'(0)=0,$ and, in general, $f^{(n)}(0)=0$ for every $n\in\mathbb{N}.$ This means that if $f$ is an eigenfunction of $T_g,$ then $f\equiv 0,$ that is, $T_g$ has no eigenvalues $\lambda\neq0,$ and since it is compact, $r(T_g)=0.$ From the first part we have that $e^{sg}\in X$ for every $s>0.$ 
\end{proof}

In our mixed norm spaces, the last result becomes:
\begin{cor}
Let $g$ be an analytic function on the unit disk. 
\begin{itemize}
\item If $g\in\mathcal{B},$ then $e^g\in H(p,\infty,\alpha)$ for some $p>0.$
\item If $g\in\mathcal{B}_0,$ then $e^g\in H(p,\infty,\alpha)$ for every $p>0.$
\end{itemize}
\end{cor}

We can also prove the converse. This result appears in \cite{Pom} in the case of the Hardy space $H^2$ as a first step to prove that $g\in BMOA.$ For $g$ analytic on the unit disk we denote by $g_\zeta$ the function $g_\zeta(z)=g(\phi_\zeta(z))-g(\zeta),$ $z\in\mathbb{D},$ with $\phi_\zeta(z)=\frac{z+\zeta}{1+\overline{\zeta}z}$ the automorphism of the disk associated with $\zeta.$

\begin{prop}
Let $X$ be a Banach space of analytic functions such that the point evaluation functionals $\delta_z$ are bounded for every $z	\in\mathbb{D}.$ If $\sup_{\zeta\in\mathbb{D}}\|e^{g_\zeta}\|_X<\infty$ then $g\in\mathcal{B}.$ In particular, if $X$ is conformal invariant and $e^g\in X$, then $g\in \mathcal B$. 
\end{prop}

\begin{proof}
Since $X$ is a Banach space and  every point evaluation functional is bounded, they are uniformly bounded on compact subsets of the unit disk (using the Banach--Steinhaus theorem). Therefore, every point evaluation of the derivative is bounded. Write $k(f)=f'(0)$. 
Taking in particular the evaluation of the derivative at zero we have 
$$|k(e^{g_\zeta})|=|\phi_\zeta'(0)g'(\phi_\zeta(0))e^{g_\zeta(0)}|=(1-|\zeta|^2)|g'(\zeta)|\leq\|k\|\|e^{g_\zeta}\|_X.$$
From here, $$\sup_{\zeta\in\mathbb{D}}(1-|\zeta|^2)|g'(\zeta)|\leq\sup_{\zeta\in\mathbb{D}}\|k\|\|e^{g_\zeta}\|_X<\infty$$ and we conclude that $g\in\mathcal{B}.$
\end{proof}

\section{Semigroups of composition operators on $H(p,\infty,\alpha)$ and $H(p,0,\alpha)$}

Before considering the ``big-Oh'' space $H(p,\infty,\alpha),$ we will study the semigroups on the ``little-oh'' space $H(p,0,\alpha),$ where polynomials are dense.

\begin{prop}
Every semigroup of analytic functions generates a strongly continuous semigroup of operators on $H(p,0,\alpha)$, but no non-trivial semigroup of analytic functions induces a uniformly continuous semigroup of composition operators on it.
\end{prop}
\begin{proof} Let $(\varphi_t)$ be a semigroup of analytic functions in the unit disk. By Proposition \ref{CO}, we have that $\limsup_{t\to 0^+}||C_t||<+\infty$. Let $\varepsilon>0$ and $r_0<1$ such that $(1-~r_0)^\alpha<~\varepsilon/4.$ Since $\varphi_t\to \varphi_0$ uniformly on compact sets, in particular in $\overline{D(0,r_0)},$ if $t$ is small enough then $M_p(r,\varphi_t-\varphi_0)<\varepsilon$ for $r\leq r_0.$ Hence, using that $M_p(r,\varphi_t-\varphi_0)\leq 2(M_p(r,\varphi_t)+M_p(r,\varphi_0))\leq 4,$ we have that
\begin{multline*}
\|\varphi_t-\varphi_0\|_{p,\infty,\alpha}=\sup_{0\leq r<1}(1-r)^\alpha M_p(r,\varphi_t-\varphi_0)\\
\begin{aligned}
&=\max\left\{\sup_{0\leq r\leq r_0}(1-r)^\alpha M_p(r,\varphi_t-\varphi_0),\sup_{r_0<r<1}(1-r)^\alpha M_p(r,\varphi_t-\varphi_0)\right\}\\&\leq\max\left\{\varepsilon\sup_{0\leq r\leq r_0}(1-r)^\alpha,4\sup_{r_0<r<1}(1-r)^\alpha\right\}=\varepsilon.
\end{aligned}
\end{multline*}
Thus, by Proposition \ref{general_theorem}, the semigroup $C_t$ is strongly continuous on $H(p,0,\alpha)$.

As in the proof of Theorem \ref{finite} we will show that the infinitesimal generator of $(C_t),$ $\Gamma(f)(z)=G(z)f'(z)$ with $G$ the generator of the semigroup $(\varphi_t),$ is not a bounded operator on $H(p,0,\alpha).$ For this, suppose $\|\Gamma f\|_{p,\infty,\alpha}=\|Gf'\|_{p,\infty,\alpha}\leq\|\Gamma\|\|f\|_{p,\infty,\alpha}$ for every $f\in H(p,0,\alpha).$ In particular, if $f_n(z)=z^n,$ $n\geq1,$ then $\|\Gamma\|\|f_n\|_{p,\infty,\alpha}\geq n\|G f_{n-1}\|_{p,\infty,\alpha}.$
Let $\delta\in (0,1)$ be such that, for all $n$, $$\sup_{\delta<r<1}(1-r)^{\alpha}r^{n-1}\geq\frac{1}{2}\sup_{0<r<1}(1-r)^{\alpha}r^{n}.$$ Hence, since the integral means are increasing functions of $r,$ 
\begin{align*}
\|\Gamma\|\|f_n\|_{p,\infty,\alpha}&\geq n\|G f_{n-1}\|_{p,\infty,\alpha}=n\sup_{0<r<1}(1-r)^\alpha M_p(r,f_{n-1}G)\\&\geq n\sup_{\delta<r<1}(1-r)^\alpha r^{n-1} M_p(r,G)\geq n\,M_p(\delta,G)\sup_{\delta<r<1}(1-r)^\alpha r^{n-1}
\\&\geq \frac{n}{2}M_p(\delta,G)\sup_{0<r<1}(1-r)^{\alpha}r^{n}=\frac{n}{2}M_p(\delta,G)\|f_n\|_{p,\infty,\alpha}.
\end{align*}
That is, $$nM_p(\delta,G)\leq 2\|\Gamma\|$$ for $n\in\mathbb{N},$ thus $M_p(\delta,G)=0$ and $G\equiv0.$
\end{proof}

Therefore $H(p,0,\alpha)=[\varphi_t,H(p,0,\alpha)],$ so, for every semigroup $(\varphi_t),$ $$H(p,0,\alpha)\subseteq[\varphi_t,H(p,\infty,\alpha)]\subseteq H(p,\infty,\alpha).$$ But the second inclusion is never an equality, as the following theorem shows.


\begin{thm}\label{nosem}
No nontrivial semigroup induces a strongly continuous semigroup of operators on $H(p,\infty,\alpha).$ In other words, $$[\varphi_t,H(p,\infty,\alpha)]\subsetneq H(p,\infty,\alpha)$$ for every semigroup of analytic functions $(\varphi_t).$
\end{thm}

To prove this theorem, by Proposition \ref{maximal} we are interested in the integral operator on $H(p,\infty,\alpha)$ and whether the subspace 
\begin{equation*}
T_\gamma(H(p,\infty,\alpha))=\{h\in H(p,\infty,\alpha):T_\gamma(f)=h\text{ for some }f\in H(p,\infty,\alpha)\}
\end{equation*} 
is dense in $H(p,\infty,\alpha).$ Differentiating we have $$T_\gamma(H(p,\infty,\alpha))=\left\{h\in H(p,\infty,\alpha):\frac{h'}{\gamma'}\in H(p,\infty,\alpha)\right\}.$$ 
If $\gamma$ is the associated g-symbol of a semigroup with Denjoy-Wolff point $b=0$ we study the density of $$E=\left\{h\in H(p,\infty,\alpha):P\, h'\in H(p,\infty,\alpha)\right\}$$ or $$E=\left\{h\in H(p,\infty,\alpha):(1-z)^2P\,h'\in H(p,\infty,\alpha)\right\},$$ for $b=1,$  where $P$ is the function with $\text{Re } P\geq0$ associated to the infinitesimal generator of the semigroup. 

First we prove a lemma that gives us information about the functions in $\overline{E}.$
For $\theta\in[0,2\pi]$ define $S_\theta$ as the Stolz region with vertex $e^{i\theta}$ and the function 
$$
\psi(\theta)=\inf_{z\in S_\theta}|P(z)|
$$ if $b=0$ and  
$$\psi(\theta)=\inf_{z\in S_\theta}|(1-z)^2P(z)|$$ if $b=1$. In both cases, we have that  $\psi(\theta)>0$ a.e. $\theta\in[0,2\pi].$ We only have to prove it for the second case. Since the function $z\mapsto \frac{1}{(1-z)^2}$ belongs to the Hardy space $H^\beta$ for $\beta <1/2$ and $1/P$ belongs to  $H^\alpha$ for $\alpha<1$ (because its real part is positive), we have that  the function $z\mapsto \frac{1}{(1-z)^2P(z)}$ belongs to the Hardy space $H^\delta$ for $\delta<1/3$. Consequently, for almost every $\theta$, 
$$\sup_{z\in S_\theta}\frac{1}{|(1-z)^2P(z)|}<+\infty.$$

\begin{lemma}
Let $h\in \overline{E}$ and $\theta\in[0,2\pi]$ such that $\psi(\theta)>0,$ then 
$$\lim_{r\to1^-}(1-r)^\alpha\left(\int_{\theta-(1-r)}^{\theta+(1-r)}|h(re^{it})|^pdt\right)^{\frac{1}{p}}=0.$$
\end{lemma}



\begin{proof}
We may assume that $b=0$ and that $\theta=0$.  If $h\in E$ then for every $r\in (0,1)$ we have
\begin{gather*}
(1-r)^\alpha\left(\int_{-(1-r)}^{1-r}\psi^p(0)|h'(re^{it})|^pdt\right)^{\frac{1}{p}}\\\leq (1-r)^\alpha\left(\int_{-(1-r)}^{1-r}|P(re^{it})|^p|h'(re^{it})|^pdt\right)^{\frac{1}{p}}\leq \|Ph'\|_{p,\infty,\alpha}=M.
\end{gather*}
From here, $$\left(\int_{-(1-r)}^{1-r}|h'(re^{it})|^pdt\right)^{\frac{1}{p}}\leq\frac{M/\psi(0)}{(1-r)^\alpha}.$$

Now, writing $$h(re^{it})=h(0)+e^{it}\int_0^rh'(\rho e^{it})\,d\rho$$ and using Minkowski's integral inequality we have that

$$\left(\int_{-(1-r)}^{1-r}|h(re^{it})|^pdt\right)^{\frac{1}{p}}\leq C+\int_0^r\frac{M/\psi(0)}{(1-\rho)^\alpha}=C+\frac{C'}{(1-r)^{\alpha-1}}.$$

Hence, $$\lim_{r\to1^-}(1-r)^\alpha\left(\int_{-(1-r)}^{1-r}|h(re^{it})|^pdt\right)^{\frac{1}{p}}\leq \lim_{r\to1^-}(1-r)^\alpha\left(C+\frac{C'}{(1-r)^{\alpha-1}}\right)=0. $$

Since $$(1-r)^\alpha\left(\int_{-(1-r)}^{1-r}|h(re^{it})|^pdt\right)^{\frac{1}{p}}$$ is bounded by the norm $\|h\|_{p,\infty,\alpha},$ the result also holds for every $h\in \overline{E}.$
\end{proof}

Thus, given any $\theta$ such that $\psi(\theta)>0$, the function $f(z)=\frac{1}{(1-e^{-i\theta}z)^{\alpha+\frac{1}{p}}},$ $z\in\mathbb{D},$ is an example of a function in $H(p,\infty,\alpha)$ that does not belong in $\overline{E},$ proving Theorem \ref{nosem}.

Now that we know $$H(p,0,\alpha)\subseteq[\varphi_t,H(p,\infty,\alpha)]\subsetneq H(p,\infty,\alpha)$$ for every semigroup $(\varphi_t),$ we want to characterize the semigroups for which $$H(p,0,\alpha)=[\varphi_t,H(p,\infty,\alpha)].$$ 

First, we deal with the case $b=0.$ It is worth noticing that, unlike the BMOA case (see \cite[Prop. 3]{BCDMPS}), the integral operator we are interested in is bounded in $H(p,\infty,\alpha)$ for every admissible function $\gamma.$ Indeed, by Proposition \ref{integralO}, $T_\gamma$ is bounded from $H(p,\infty,\alpha)$ to itself if and only if $\gamma\in\mathcal{B},$ that is, $\sup_{z\in\mathbb{D}}(1-|z|)\frac{1}{|P(z)|}<\infty,$ and this is true since every such $P$ induced by a semigroup of analytic functions has positive real part.

\begin{thm} Let $(\varphi_t)$ be a semigroup with Denjoy-Wolff point $b\in\mathbb{D}.$ Then 
$$H(p,0,\alpha)=[\varphi_t,H(p,\infty,\alpha)]\Leftrightarrow \gamma\in\mathcal{B}_0.$$ Moreover, if $\gamma\not\in \mathcal{B}_0$ then the space $[\varphi_t,H(p,\infty,\alpha)]$ contains a subspace isomorphic to $l_\infty.$
\end{thm}

\begin{proof}
By Proposition \ref{maximal}, $$[\varphi_t, H(p,\infty,\alpha)]=\overline{H(p,\infty,\alpha)\cap(T_\gamma(H(p,\infty,\alpha))\oplus \mathbb{C})},$$ 
where $T_\gamma$ is the integral operator with symbol $\gamma(z)=\int_0^z\frac{1}{P(\zeta)}d\zeta.$ 
Since $T_\gamma$ is bounded on $H(p,\infty,\alpha)$ we have $$T_\gamma(H(p,\infty,\alpha))\oplus \mathbb{C}\subset H(p,\infty,\alpha)$$ and therefore $$\overline{H(p,\infty,\alpha)\cap(T_\gamma(H(p,\infty,\alpha))\oplus \mathbb{C})}=\overline{T_\gamma(H(p,\infty,\alpha))\oplus \mathbb{C}}.$$ From here, $$H(p,0,\alpha)=[\varphi_t,H(p,\infty,\alpha)]=\overline{T_\gamma(H(p,\infty,\alpha))\oplus \mathbb{C}}$$ if and only if $T_\gamma(H(p,\infty,\alpha))\subseteq H(p,0,\alpha),$ and, by Proposition \ref{integral_compact}, this is equivalent to $\gamma\in\mathcal{B}_0.$ 
\end{proof}

Now, for $b=1,$ recall that $$E=\left\{h\in H(p,\infty,\alpha):(1-z)^2P\,h'\in H(p,\infty,\alpha)\right\}.$$ As before, we have that $[\varphi_t,H(p,\infty,\alpha)]=\overline{E}$ for $(1-z)^2P(z)$ the generator of the semigroup $(\varphi_t).$ 

\begin{thm}
For every semigroup of analytic functions with Denjoy-Wolff $b\in\mathbb{T}$ the set $\overline{E}$ contains a copy of $l_\infty.$ Consequently, $[\varphi_t,H(p,\infty,\alpha)]\supsetneq H(p,0,\alpha).$
\end{thm}

Several auxiliary results are needed to prove this theorem. First, we define the space $$X_{\nu,\beta}=\left\{\sum_{n=1}^\infty a_nf_n:(a_n)\in l_\infty\right\},$$ where $f_n(z)=\delta_n^\nu(1+\delta_n-z)^\beta,$ $\delta_n=K^{-n}$ for $K$ big enough, $\nu>0$ and $\beta p<-1.$ 

The first lemma relates the space $X_{\nu,\beta}$ with $H(p,\infty,\alpha)$ for some $\nu$ depending on $p$ and $\alpha.$

\begin{lemma}\label{nu}
Let $F(z)=\sum_{n=1}^\infty |f_n(z)|$ then $$(1-r)^\alpha M_p(r,F)\leq C$$ for every $r$ and for $\nu=-(\alpha+\beta+\frac{1}{p}).$
\end{lemma}

\begin{proof} Adapting the proof of 
\cite[Lemma in Section 4.6]{Dur}, we can see that there exists a constant $C$ such that $$M_p(r,f_n)\leq C \delta_n^\nu (1+\delta_n-r)^{\beta+\frac{1}{p}}.$$

Let $l\in\mathbb{N}$ such that $1-r\in(\delta_{l+1},\delta_l].$ Then, since $1+\delta_n-r\approx \delta_n$ for $n\leq l$ and $1+\delta_n-r\approx 1-r$ for $n>l,$ we have 
\begin{align*}
(1-r)^\alpha M_p(r,F)&\leq \sum_{n=1}^\infty(1-r)^\alpha M_p(r,f_n)\leq C\sum_{n=1}^\infty(1-r)^\alpha\delta_n^\nu(1+\delta_n-r)^{\beta+\frac{1}{p}}\\&\leq C\sum_{n=1}^l(1-r)^\alpha\delta_n^{\nu+\beta+\frac{1}{p}}+C\sum_{n=l+1}^\infty(1-r)^{\alpha+\beta+\frac{1}{p}}\delta_n^\nu\\&\leq C(1-r)^\alpha \delta_l^{-\alpha}+C(1-r)^{-\nu}\delta_{l+1}^\nu\leq C',
 \end{align*}
given that both $\alpha$ and $\nu$ are positive.
\end{proof}

The first property we will need about this spaces is the following:

\begin{prop}\label{xnu}
Define \begin{align*}
\Phi:\,l_\infty&\to X_{\nu,\beta}\\
(a_n)&\to \sum_{n=1}^\infty a_nf_n.
\end{align*} 
Then $\Phi$ is an isomorphism between $l_\infty$ and $X_{\nu,\beta}$ with the norm of $H(p,\infty,\alpha)$ if $\nu=-(\alpha+\beta+\frac{1}{p}).$
\end{prop}

\begin{proof}
Following the steps of Theorem \ref{copy_l_infty} it is easy to see using the previous lemma that $\Phi$ is well defined and bounded from $l_\infty$ to $X_{\nu,\beta}.$ We only need to show that, for some $C>0$, given $\textbf{a}=(a_n)\in l_\infty$, $C\|\Phi {\bf a}\|_{p,\infty,\alpha}\geq\|\bf a\|_\infty.$ Let $\|\textbf{a}\|_\infty=1$ and $l\in\mathbb{N}$ such that $|a_l|>\frac{1}{2}$ and $\delta_l\approx (1-r_l).$ By definition of $\Phi {\bf a},$ for $|t|<\delta_l$ we have, as in the proof of the previous lemma, 
\begin{align*}
|\Phi {\bf a}(r_le^{it})|&
\geq|a_l||f_l(r_le^{it})|-\sum_{n=1}^{l-1}|a_n||f_n(r_le^{it})|-\sum_{n=l+1}^\infty|a_n||f_n(r_le^{it})|\\&
\geq \frac{1}{2}\delta_l^\nu|1+\delta_l-r_le^{it}|^\beta-\sum_{n=1}^{l-1}\delta_n^\nu|1+\delta_n-r_le^{it}|^\beta-\sum_{n=l+1}^{\infty}\delta_n^\nu|1+\delta_n-r_le^{it}|^\beta\\&
\geq \frac{C}{2}\delta_l^\nu\delta_l^\beta-C'\sum_{n=1}^{l-1}\delta_n^{\nu+\beta}-C''\sum_{n=l+1}^{\infty}\delta_n^\nu\delta_l^\beta\\&
\geq \frac{C}{2}\delta_l^{\nu+\beta}-C'\delta_{l-1}^{\nu+\beta}-C''\delta_{l+1}^\nu\delta_l^\beta
\\&\geq  \frac{C}{2}\delta_l^{\nu+\beta}-C'K^{\nu+\beta}\delta_l^{\nu+\beta}-C''K^{-\nu}\delta_l^{\nu+\beta}
\geq A\delta_l^{\nu+\beta}\geq A'(1-r_l)^{\nu+\beta}.
\end{align*}
Then, since $|t|<\delta_l\approx 1-r_l$
$$M_p^p(r_l,\Phi{\bf a})\geq\int_{|t|<\delta_l}|\Phi {\bf a}(r_le^{it})|^p\frac{dt}{2\pi}\geq \frac{A'}{\pi}(1-r_l)^{(\nu+\beta)p}\delta_l\approx C(1-r_l)^{(\nu+\beta)p+1}$$
it follows that
\begin{align*} \|\Phi {\bf a}\|_{p,\infty,\alpha}&\ge
(1-r_l)^{\alpha}M_p(r_l,\Phi{\bf a}) \\&\geq C(1-r_l)^\alpha(1-r_l)^{\nu+\beta+\frac{1}{p}}= C.
\end{align*}
\end{proof}

Our interest on the space $X_{\nu,\beta}$ comes from the following two propositions. In the first one we prove that, if $\alpha>1,$ the set $E$ contains a copy of $l^\infty.$

\begin{prop}
Let $\nu=-(\alpha+\beta+\frac{1}{p})$ and $\alpha>1,$ then $X_{\nu,\beta}\subseteq E.$
\end{prop}

\begin{proof}
For $f\in X_{\nu,\beta}$ we already know that $f\in H(p,\infty,\alpha)$ for the given parameter $\nu$ (Lemma \ref{nu}), so we only need to show $(1-z)^2Pf'\in  H(p,\infty,\alpha).$
Since $$f'(z)=-\sum_{n=1}^\infty a_n\delta_n^\nu\beta(1+\delta_n-z)^{\beta-1},$$ it is clear that $f'\in X_{\nu,\beta-1}.$ Multiplying by $(1-z)^2$ we have that $$|f'(z)(1-z)^2|\leq C\sum_{n=1}^\infty\delta_n^\nu|1+\delta_n-z|^{\beta-1}|1-z|^2\leq C\sum_{n=1}^\infty\delta_n^\nu|1+\delta_n-z|^{\beta+1}.$$ 
Hence, by Lemma \ref{nu}, $(1-z)^2 f'\in H(p,\infty,\alpha -1)$. Moreover, since $P$ has positive real part,  $P\in H(\infty,\infty,1)$ and then $(1-z)^2Pf'\in  H(p,\infty,\alpha).$ 
\end{proof}

The next proposition allows us to take away the condition on $\alpha$ by taking the closure on $E.$ Since $\overline{E}=[\varphi_t,H(p,\infty,\alpha)],$ we finally get that, if $b=1,$ $[\varphi_t,H(p,\infty,\alpha)]$ contains a copy of $l^\infty,$ and therefore it can never be $H(p,0,\alpha).$ First we need the following lemma.

\begin{lemma}
For every $f\in X_{\nu,\beta}$ and $\theta\in (0,\alpha)$ $$f'(z)(1-z)[(1-z)P(z)]^{\theta}\in H(p,\infty,\alpha).$$
\end{lemma}

\begin{proof}
As in the proof of the last Proposition, we have that, for $f\in X_{\nu,\beta},$ $$|f'(z)(1-z)^{1+\theta}|\leq C\sum_{n=1}^\infty\delta_n^\nu|1+\delta_n-z|^{\beta-1}|1-z|^{1+\theta}\leq C\sum_{n=1}^\infty\delta_n^\nu|1+\delta_n-z|^{\beta+\theta}$$ so, by Lemma \ref{nu}, $(1-z)^{1+\theta} f'\in H(p,\infty,\alpha-\theta).$ Since $P^\theta\in H(\infty,\infty,\theta),$ we have $f'(z)(1-z)^{1+\theta}P(z)^{\theta}\in H(p,\infty,\alpha).$
\end{proof}

\begin{prop}\label{xnucierre}
If $\nu=-(\alpha+\beta+\frac{1}{p})$ then $X_{\nu,\beta}\subseteq \overline{E}.$
\end{prop}

\begin{proof}
For $f\in X_{\nu,\beta}$ we define $$h_n=f(0)+V\left(\frac{nf'}{n+\psi}\right),$$ where $V$ is the Volterra operator and $\psi(z)=[(1-z)P(z)]^{1-\theta}$ for $0<\theta<1.$ 

Since $$\psi(\mathbb{D})\subseteq\Delta=\{w:\text{Arg}\,w\in(-\pi(1-\theta),\pi(1-\theta))\},$$ we have that there exists a constant $M$ such that $\big|\frac{w}{n+w}\big|<M$ for every $w\in\Delta$ and $n\in\mathbb{N}.$ Now, since $\frac{n}{n+\psi(z)}=1-\frac{\psi(z)}{n+\psi(z)}$ we have $\big|\frac{n}{n+\psi(z)}\big|<M+1$ for every $z\in\mathbb{D}$ and therefore it is a bounded function. This allows us to show that $h'_n=\frac{nf'}{n+\psi}\in  H(p,\infty,\alpha+1),$ since $f'\in X_{\nu,\beta-1}\subseteq H(p,\infty,\alpha+1)$ for every $f\in X_{\nu,\beta}$, and $\frac{n}{n+\psi}\in H^\infty.$ From here, $h_n\in H(p,\infty,\alpha).$ 

Moreover, $h_n\in E$ for every $n\in\mathbb{N}$. Indeed, by the definition of $h_n$ and $\psi$ we have $$P(z)(1-z)^2h'_n(z)=P(z)(1-z)^2\frac{nf'}{n+\psi(z)}=f'(z)(1-z)^{1+\theta}P(z)^{\theta}\psi(z)\frac{n}{n+\psi(z)}.$$ By the previous Lemma, for every $f\in X_{\nu,\beta}$ and $\theta\in (0,\alpha)$ $$f'(z)(1-z)^{1+\theta}P(z)^{\theta}\in H(p,\infty,\alpha),$$ and recalling that $\frac{\psi}{n+\psi}\in H^\infty,$ we get $P(z)(1-z)^2h'_n(z)\in H(p,\infty,\alpha),$ that is, $h_n\in E.$

To prove $X_{\nu,\beta}\subseteq \overline{E}$ we are going to show that $h_n\to f\text{ in } H(p,\infty,\alpha).$ Taking derivatives, this is equivalent to $h_n'\to f'$ in $H(p,\infty,\alpha+1),$ and by definition of $h_n,$ to $$\left(\frac{n}{n+\psi}-1\right)f'=\left(\frac{\psi}{n+\psi}\right)f'\to 0$$ in $ H(p,\infty,\alpha+1)$

Rephrasing, in the following we want to prove that $\left(\frac{\psi}{n+\psi}\right)g\to 0$ in $H(p,\infty,\alpha)$ for $g\in X_{\nu,\beta}.$ 
Define the sets $$A_n=\{z=re^{it}\in\mathbb{D}:|t|\leq n(1-r), |t|\leq\pi\}$$ and $$B_n=\{z=re^{it}\in\mathbb{D}:n(1-r)<|t|\leq\pi\}$$ and the associated functions $c_n=\frac{\psi}{n+\psi}g\chi_{A_n}$ and $d_n=\frac{\psi}{n+\psi}g\chi_{B_n},$ so $$\frac{\psi}{n+\psi}g=c_n+d_n$$ and 
\begin{equation}\label{final}
\left\|\frac{\psi}{n+\psi}g\right\|_{p,\infty,\alpha}\leq \sup_{0\leq r<1}(1-r)^\alpha[M_p(r,c_n)+M_p(r,d_n)].
\end{equation}

First, suppose $z=re^{it}\in A_n,$ then $|t|\leq n(1-r)$ and $|1-z|\leq |t|+1-r\leq (n+1)(1-r).$ Now, since $\text{Re}\,P>0,$ let $$Q(z)=\frac{P(z)-i\,\text{Im}\,P(0)}{\text{Re}\,P(0)},$$ then $\text{Re}\,Q(z)>0$ and $Q(0)=1,$ so $|Q(z)|\leq\frac{2}{1-|z|}$ and $|P(z)|\leq\frac{2\text{Re}\,P(0)+|\text{Im}\,P(0)|}{1-|z|}.$ Using these last two observations, 
$$|(1-z)P(z)|\leq \left(2\text{Re}\,P(0)+|\text{Im}\,P(0)|\right)(n+1),$$ so $$|\psi(z)|=|(1-z)^\theta P^\theta(z)|\leq \left(2\text{Re}\,P(0)+|\text{Im}\,P(0)|\right)^\theta (n+1)^\theta<\frac{n}{2}$$ for $n$ large enough.
From here $$\left|\frac{\psi}{n+\psi}\right|\leq\frac{\left(2\text{Re}\,P(0)+|\text{Im}\,P(0)|\right)^\theta (n+1)^\theta}{n-\frac{n}{2}}\leq C(n+1)^{\theta-1}=\alpha_n,$$ with $\alpha_n\to0$ when $n\to\infty$ (recall that $0<\theta<1$). Therefore, $$\lim_{n\to\infty}\sup_{0\leq r<1}(1-r)^\alpha M_p(r,c_n)\leq \lim_{n\to\infty}\alpha_n\|g\|_{p,\infty,\alpha}=0.$$

Now, take $n\geq 7.$ If $z=re^{it}\in B_n,$ then $1-r<\frac{\pi}{n}$, so $r>1/2.$ Recall that $d_n=\frac{\psi}{n+\psi}\sum_{k=1}^\infty a_kf_k \chi_{B_n}$ with $f_k(z)=\delta_k^\nu(1+\delta_k-z)^\beta$, $\delta_k=K^{-k}$ and $\beta<0.$ For $z\in B_n$ we have $|1+\delta_k-z|\geq|r-z|=r|1-e^{it}|\geq \frac{r}{\pi}|t|\geq \frac{|t|}{2\pi}.$ Therefore $|f_k(z)|\leq \delta_k^\nu\left(\frac{|t|}{2\pi}\right)^{\beta}$ and 
\begin{align*} M_p(r,f_k\chi_{B_n})&\leq\delta_k^\nu\left(2\int_{n(1-r)}^\pi\left(\frac{t}{2\pi}\right)^{\beta p}\frac{dt}{2\pi}\right)^{\frac{1}{p}}\\&\leq\delta_k^\nu\left(2\int_{n(1-r)}^\infty\left(\frac{t}{2\pi}\right)^{\beta p}\frac{dt}{2\pi}\right)^{\frac{1}{p}}\leq C\delta_k^\nu\left(n(1-r)\right)^{\beta+\frac{1}{p}}.
\end{align*}

Now we take $l\in\mathbb{N}$ such that $(1-r)K^l\in\left(\frac{1}{\sqrt{n}},\frac{K}{\sqrt{n}}\right]$ (that means, $\left((1-r)K^l\right)^\alpha\leq \left(K/\sqrt{n}\right)^\alpha$ and $\left((1-r)K^l\right)^{-\nu}\leq \left(1/\sqrt{n}\right)^{-\nu}$). The integral mean of $d_n$ can be bound as 

\begin{align*}
M_p(r,d_n)&\leq \left|\frac{\psi}{n+\psi}\right|\|a_n\|_{l^\infty}\sum_{k=1}^\infty M_p(r,f_k\chi_{B_k})\\&
\leq(M+1)\|a_n\|_{l^\infty}\left[\sum_{k=1}^{l-1}C\delta_k^{-\alpha}
+\sum_{k=l}^\infty C'\delta_k^\nu\left(n(1-r)\right)^{\beta+\frac{1}{p}}\right]\\&
\leq(M+1)\|a_n\|_{l^\infty}\left(C K^{l\alpha}
+C'n^{\beta+\frac{1}{p}} K^{-l\nu}\left(1-r\right)^{\beta+\frac{1}{p}}\right)\\&
=\frac{(M+1)\|a_n\|_{l^\infty}}{(1-r)^\alpha}\left(C K^{l\alpha}(1-r)^\alpha+C'n^{\beta+\frac{1}{p}}K^{-l\nu}(1-r)^{-\nu}\right)\\&
\leq\frac{(M+1)\|a_n\|_{l^\infty}}{(1-r)^\alpha}\left(C\frac{K^{\alpha}}{n^{\alpha/2}}+C'n^{\beta+\frac{1}{p}}n^{\nu/2}\right)
\\&=\frac{(M+1)\|a_n\|_{l^\infty}}{(1-r)^\alpha}\left(CK^{\alpha}n^{-\alpha/2}+C'n^{-\alpha-\nu/2}\right).
\end{align*}
From here clearly $\sup_{0\leq r<1}(1-r)^\alpha M_p(r,d_n)\to0$ when $n\to\infty,$ and, by (\ref{final}), $$\left\|\frac{\psi}{n+\psi}g\right\|_{p,\infty,\alpha}\leq \sup_{0\leq r<1}(1-r)^\alpha[M_p(r,c_n)+M_p(r,d_n)]\to0$$ when $n\to\infty.$ This finishes the proof.
\end{proof}


\section*{Acknowledgements}
This work was partly developed during the first author’s two different research stays at Universidad de Sevilla and IMUS. She wants to use this opportunity to thank everyone there for their hospitality.

\bibliographystyle{amsplain}

\begin{thebibliography}{20}
\bibitem{aj} P. Ahern and M. Jevti\'c, \rm{Duality and multipliers for mixed norm spaces}, \textit{Michigan Math. J. } \textbf{ 30} (1983), 53--64. 
\bibitem{ASHp} 
A. Aleman and A. G. Siskakis, \rm{An integral operator on $H^p$}, \textit{Complex Variables Theory Appl. } \textbf{28} (1995), 149--158.
\bibitem{ASAp}
A. Aleman and A. G. Siskakis, \rm{Integration operators on Bergman spaces}, \textit{Indiana Univ. Math. J. } \textbf{46} (1997), 337--356.
\bibitem{Ar}
I. Ar\'evalo, A characterization of the inclusions between mixed norm spaces, \textit{J. Math. Anal. Appl.} \textbf{429} (2015), 942--955.
\bibitem{BCHMP}
M. Basallote, M. D. Contreras, C. Hern\'andez-Mancera, M. J. Mart\'in and  P. J. Pa\'ul,
Volterra operators and semigroups in weighted Banach spaces of analytic functions, \textit{Collect. Math. } \textbf{65} (2014), 233--249. 
\bibitem{BerPor}
E. Berkson, H. Porta, \rm{Semigroups of analytic functions and composition operators}, \textit{Michigan Math. J.} \textbf{25} (1978), 101--115.
\bibitem{Blasco} O. Blasco, \rm{Multipliers on spaces of analytic functions}, \textit{Canad. J. Math.}\textbf{ 47} (1995), 44-64.
\bibitem{BCDMPS}
O. Blasco, M. D. Contreras, S. D{\'{\i}}az-Madrigal, J. Mart{\'{\i}}nez, M. Papadimitrakis and A. G. Siskakis, \rm{Semigroups of composition operators and integral operators in spaces of analytic functions}, \textit{Ann. Acad. Sci. Fenn. Math. } \textbf{38} (2013), 67--89. 
\bibitem{BCDMS}
O. Blasco, M. D. Contreras, S. D{\'{\i}}az-Madrigal, J. Mart{\'{\i}}nez and A. G. Siskakis, \rm{Semigroups of composition operators in BMOA and the extension of a theorem of Sarason}, \textit{Integral Equations Operator Theory } \textbf{61} (2008), 45--62. 
\bibitem{BKV} S. M. Buckley, P. Koskela and D. Vukoti\'c, \rm{Fractional integration, differentiation, and weighted Bergman spaces}, \textit{Math. Proc. Cambridge Philos. Soc.}\textbf{ 126} (1999), 369-385.
\bibitem{CGP} C. Chatzifountas, D. Girela, J. A. Pel\'aez, \rm{A generalized Hilbert matrix acting on Hardy spaces}, \textit{J. Math. Anal. Appl.}\textbf{ 413} (2014), 154--168. 
\bibitem{CPPR} M. D. Contreras, J. A. Pel\'aez, Ch. Pommerenke and J. R\"atty\"a, \rm{Integral operators mapping into the space of bounded analytic functions}, \textit{J. Funct. Anal.} \textbf{271} (2016), 2899-2943.
\bibitem{Dur}
P. L.~Duren, \textit{Theory of $H^p$ Spaces\/}, Pure and Applied
Mathematics, Vol.~\textbf{38}, Second edition, Dover, Mineola, New
York 2000.
\bibitem{Flett} T. M. Flett, \rm{The dual of an inequality of Hardy and Littlewood and some related inequalities}, \textit{J. Math. Anal. Appl.} \textbf{ 38} (1972), 746-765.
\bibitem{Flett2} T. M. Flett, \rm{Lipschitz spaces of functions on the circle and the disk}, \textit{J. Math. Anal. Appl.} \textbf{ 39} (1972), 125--158.
\bibitem{Gab} S. Gadbois, \rm{Mixed-norm generalizations of Bergman spaces and duality}, \textit{Proc. Amer. Math. Soc.}\textbf{ 104} (1988), 1171--1180.
\bibitem{GGPS} P. Galanopoulos, D. Girela, J. A. Pel\'aez, A. G. Siskakis, \rm{Generalized Hilbert operators}, \textit{Ann. Acad. Sci. Fenn. Math.}\textbf{ 39} (2014), 231--258.
\bibitem{HL} G. H. Hardy and J. E. Littlewood, \rm{Some properties of fractional integrals. II.}, \textit{Math. Z.}\textbf{ 34} (1932), 403-439.

\bibitem{JVA} M. Jevti\'c, D. Vukoti\'c and M. Arsenovi\'c, \textit{Taylor Coefficients and Coefficient Multipliers of Hardy and Bergman-Type Spaces}, RSME Springer Series, Vol.~\textbf{2}, Springer, Cham, Switzerland 2016.

\bibitem{MP} M. Mateljevi\'c and M. Pavlovi\'c, \rm{$L^p$-behavior of the integral means of analytic functions}, \textit{Studia Math. }\textbf{ 77} (1984), 219--237. 
\bibitem{PR13} J. A. Pel\'aez, J. R\"atty\"a, \rm{Generalized Hilbert operators on weighted Bergman spaces}, \textit{Adv. Math. } \textbf{ 240} (2013), 227--267.
\bibitem{PS} J. A. Pel\'aez, D. Seco, \rm{Schatten classes of generalized Hilbert operators} arXiv:1510.05455.
\bibitem{Pom}
Ch. Pommerenke, \rm{Schlichte Funktionen und analytische Funktionen von beschr\"ankter mittlerer Oszillation}, \textit{Comment. Math. Helv. } \textbf{52} (1977), 591--602.
\bibitem{SisBergman} 
A. Siskakis, \rm{Semigroups of composition operators in Bergman spaces}, \textit{Bull. Austral. Math. Soc. } \textbf{35} (1987), 397--406.
\bibitem{SisDirichlet} 
A. Siskakis, \rm{Semigroups of composition operators on the Dirichlet space}, \textit{Results Math. } \textbf{30} (1996), 165--173. 
\bibitem{SisSem} 
A. Siskakis, \rm{Semigroups of composition operators on spaces of analytic functions, a review}, \textit{Contemp. Math. } \textbf{232} (1990), 229--252.
\bibitem{Sledd} W. T. Sledd, \rm{Some results about spaces of analytic functions introduced by Hardy and Littlewood}, \textit{J. London Math. Soc.}\textbf{ 9} (1974/75), 328--336.
\bibitem{Yos}
K. Yosida, \textit{Functional Analysis\/}, Classics in Mathematics, Springer-Verlag, Berlin, 1995.


\end{thebibliography}

\end{document}